\documentclass[11pt]{amsart}

\usepackage[margin=1.4in]{geometry}
\usepackage[colorlinks,citecolor=cyan,linkcolor=blue]{hyperref}

\usepackage{amsmath}
\usepackage{amsthm}
\usepackage{amssymb}
\usepackage{thmtools}
\usepackage[all,cmtip]{xy}
\usepackage{enumerate}
\usepackage{enumitem}

\newcommand{\ep}{
\epsilon
}
\newcommand{\mc}[1]{
\mathcal{#1}
}
\newcommand{\mb}[1]{
\mathbb{#1}
}

\newtheorem*{thm*}{Theorem}
\newtheorem*{cor*}{Corollary}
\newtheorem*{pro*}{Proposition}

\newtheorem{thmA}{Theorem}

\newtheorem{corA}[thmA]{Corollary}

\newtheorem{propA}[thmA]{Proposition}

\newtheorem{thm}{Theorem}[section]

\newtheorem{cor}[thm]{Corollary}
\newtheorem{lemma}[thm]{Lemma}
\newtheorem{lem}[thm]{Lemma}
\newtheorem{pro}[thm]{Proposition}

\theoremstyle{definition}
\newtheorem{claim}[thm]{Claim}

\newtheorem*{defi*}{Definition}
\newtheorem{dfn}[thm]{Definition}

\author[Sisto]{Alessandro Sisto}
	\address{Maxwell Institute and Department of Mathematics, Heriot-Watt University,     Edinburgh, UK}
	\email{a.sisto@hw.ac.uk}

\author[Viaggi]{Gabriele Viaggi}
        \address{Department of Mathematics, University of Pisa, Pisa, Italy}
        \email{gabriele.viaggi@unipi.it}

\title{Distinguishing Gromov-Thurston manifolds using algebraic Dehn fillings}

\begin{document}

\maketitle

\begin{abstract}
We develop criteria to distinguish the homotopy types of Gromov-Thurston manifolds. Our approach is based on a description of their fundamental groups as virtual Dehn fillings of relatively hyperbolic groups.   
\end{abstract}

\section{Introduction}

The first known examples of closed manifolds admitting a Riemannian metric of pinched negative curvature but no locally symmetric one were constructed by Gromov and Thurston \cite{GT} in a way that we now recall. The input is what we will refer to here as a \emph{GT-pair}.

\begin{defi*}[GT-Pair]
A {\em GT-pair} of dimension $d$ and {\em wall} $W$ is a pair $(M,B)$ where $M$ is a closed orientable hyperbolic $d$-manifold and $B\subset M$ is a codimension 2 totally geodesic connected submanifold $B$ which is the boundary of a compact codimension 1 orientable totally geodesic submanifold $W$.     
\end{defi*}

Gromov and Thurston's examples are cyclic ramified coverings of $M$ branched over $B$ of some degree $k$, which we denote here by ${\rm GT}(M,B,k)$ and call them {\em GT-manifolds}. These manifolds turned out to possess many interesting topological and geometric features beyond the original motivation, and they have been studied quite extensively, see for instance Kapovich \cite{Kapovich}, Giralt \cite{Giralt}, Fine-Premoselli \cite{FP}, Monclair-Schlenker-Tholozan \cite{MST}, Hamenstädt-Jäckel \cite{HJ}, Guenancia-Hamenstädt \cite{GuHa}, Lafont-Minemyer \cite{LaMi}, Avramidi-Okun-Schreve \cite{L2}, Hamenstädt \cite{H}.

Despite this, the classifications of GT-manifolds up to homotopy equivalence, commensurability, or quasi-isometry are all poorly understood. One of the key issues is the absence of an underlying locally symmetric structure and the lack of suitable analogues of Mostow rigidity. In this article, we focus on the study of homotopy types of GT-manifolds.

First, we consider GT-manifolds coming from the {\em same} GT-pair. In even dimension, an Euler characteristic trick shows that different degrees yield different homotopy types (we discuss the Euler characteristic and volume approaches later on in the introduction). However, in odd dimension the situation is much more mysterious. For instance, prior to this paper, it was not known whether the $k$- and $(k+1)$-degree branched covers can be homotopy equivalent to each other. We show that GT-manifolds coming from the same GT-pair can only be homotopy-equivalent if their branching degrees satisfy an arithmetic constraint, if they have high-degree branching, in particular addressing the aforementioned case of $k$- and $(k+1)$-degree. 

The following is a simplified version of Theorem \ref{thmA:distinguish}.

\medskip

\begin{corA}
\label{cor:coprime}
Let $(M,B)$ be a {\rm GT-pair} of dimension $d\ge 3$, and let $M_k={\rm GT}(M,B,k)$. Then there exists $I>0$ such that whenever $k,k'\in\mb{N}$ are coprime and $k,k'>I$, we have that $M_k$ is not homotopy equivalent to $M_{k'}$.
\end{corA}

\medskip

In order to state the more precise version, we give a preliminary definition.

\begin{defi*}[$I$-Unrelated]
Let $I>0$ be an integer. We say that $k,k'\in\mb{N}$ are {\em $I$-unrelated} if at least one between $k/{\rm gcd}(k,k')$ and $k'/{\rm gcd}(k,k')$ is $>I$. 
\end{defi*}

For example, if $k,k'$ are $>I$ and coprime, then they are $I$-unrelated. Corollary \ref{cor:coprime} is an immediate consequence of the following.

\medskip

\begin{restatable}{thmA}{distinguish}
\label{thmA:distinguish}
Let $(M,B)$ be a {\rm GT-pair} of dimension $d\ge 3$, and let $M_k={\rm GT}(M,B,k)$. There exists $I>0$ such that whenever $k,k'\in\mb{N}$ are $I$-unrelated, we have that $M_k$ is not homotopy equivalent to $M_{k'}$.
\end{restatable}

\medskip

The obstruction we use to distinguish GT-manifolds is, roughly, the number of injections from a certain hyperbolic group into their fundamental group. Said hyperbolic group is in fact the fundamental group of the complement of a wall. Specifically, Theorem \ref{thmA:distinguish} is a consequence of the following result.

\medskip

\begin{restatable}{thmA}{main}
\label{thmA:main}
Let $(M,B)$ be a {\rm GT-pair} of dimension $d\ge 3$, with wall $W$, and let $M_k={\rm GT}(M,B,k)$. There exists an integer $I>0$ such that for all $k\ge 2$ the following holds. The number of conjugacy classes of injections $\pi_1(M-W)\to \pi_1(M_k)$ is $Nk$ for some $N\leq I$.
\end{restatable}

\medskip

The strategy of proof for Theorem \ref{thmA:main} is inspired by \cite{DG18}. We give an outline of the argument at the end of the introduction. Given Theorem \ref{thmA:main}, the proof of Theorem \ref{thmA:distinguish} goes as follows. If $\pi_1(M_k)\simeq\pi_1(M_{k'})$ then the number of conjugacy classes of injections of $\pi_1(M-W)$ in $\pi_1(M_k)$ and $\pi_1(M_{k'})$ agree. However, by Theorem \ref{thmA:main}, there exists $I>0$ such that the number of conjugacy classes of injections is $Nk$ for $\pi_1(M_k)$ and $N'k'$ for $\pi_1(M_{k'})$ with $N,N'\le I$. If $k,k'$ are $I$-unrelated we can never have $Nk=N'k'$. Theorem \ref{thmA:distinguish} follows.

As we already mentioned, Euler characteristic and (simplicial) volume provide two methods to distinguish GT-manifolds coming from the {\em same} GT-pair. However, they yield limited information when the GT-manifolds come from {\em different} GT-pairs. We address this by studying the algebraic structure of ${\rm Out}(\pi_1(M_k))$, which we also believe to be of independent interest. Intuitively, ${\rm Out}(\pi_1(M_k))$ should contain a cyclic subgroup of order $k$ (the deck group of $M_k\to M$) plus some uniformly bounded noise that comes from the lifts of some symmetries of the base $M$. Again relying on Theorem \ref{thmA:main}, we prove the following.

\medskip

\begin{restatable}{thmA}{Out}
\label{thmA:Out}
Let $(M,B)$ be a {\rm GT-pair} of dimension $d\ge 3$ and let $M_k={\rm GT}(M,B,k)$.
There exists an integer $I>0$ such that, for all $k\ge 2$, $|{\rm Out}(\pi_1(M_k))|$ divides $I^kk$.
\end{restatable}

\medskip

As an application, we use Theorem \ref{thmA:Out} to distinguish the homotopy types of Gromov-Thurston manifolds with non-isometric branching loci.

\medskip

\begin{restatable}{thmA}{Sylow}
\label{thmA:Sylow}
Let $(M,B)$, $(M',B')$ be {\rm GT-pairs} of dimension $d\ge 5$, and let $M_k={\rm GT}(M,B,k)$, $M'_k={\rm GT}(M',B',k)$. Suppose that $B$ is not isometric to $B'$. Then there exists $I>0$ such that the following holds. If $k,k'>I$ and $k$ has all prime factors $>I$, then $M_k$ is not homotopy equivalent to $M'_{k'}$.
\end{restatable}

\medskip

All our results rely on the key fact that fundamental groups of Gromov-Thurston manifolds for the GT-pair $(M,B)$ are virtual Dehn fillings of $\pi_1(M-B)$ plus the fact that $\pi_1(M-B)$ is relatively hyperbolic (with respect to natural peripheral subgroups) by a result of Belegradek \cite{Bel}. A simplified version is the following:

\begin{propA}[\cite{Bel} and Lemma \ref{lem:van kampen}]
Let $(M,B)$ be a {\rm GT-pair}, and let $M_k={\rm GT}(M,B,k)$. Denote by $U\subset M$ a tubular neighborhood of $B$. Then \begin{itemize}
    \item{$\pi_1(M-B)$ is hyperbolic relative to $\pi_1(U-B)$.}
    \item{$\pi_1(M_k)$ is isomorphic to an index-$k$ subgroup of $\pi_1(M-B)/\langle\langle \gamma^k\rangle\rangle$ where $\gamma\in\pi_1(U-B)$ is the meridian.}
\end{itemize}
\end{propA}

The proposition allows us to use various tools from geometric group theory, in particular limiting actions on trees; more details below.

Having stated our main results, we conclude the introduction by first reviewing the Euler characteristic and volume methods and then giving the outline of our approach together with the organisation of the article.

\subsection*{Euler characteristic and volumes}
Let us discuss first the Euler characteristic.

A computation shows that the Euler characteristic of $M_k={\rm GT}(M,B,k)$ is given by 
\[
\chi(M_k)=k(\chi(M)-\chi(B))+\chi(B).
\]
While this is trivial in odd dimension (zero-equals-zero), it is very useful in even dimension when combined with the Chern-Gauss-Bonnet Theorem which says that ${\rm vol}(M)$ and ${\rm vol}(B)$ are universal positive multiples of $(-1)^{d/2}\chi(M)$ and $(-1)^{d/2-1}\chi(B)$. In particular, as the volume is always positive, it implies that $\chi(M)-\chi(B)\neq 0$ and, hence, that $\chi(M_k)$ grows linearly in $k$. From this, we can conclude that, in even dimensions, if $k\neq k'$ then $M_k$ cannot be homotopy equivalent to $M_k'$.

In a similar fashion, in an arbitrary dimension, one could try to replace the Euler characteristic with the simplicial volume $||M_k||$ (we refer to \cite[Chapter 7]{Frigerio} for the definition of this topological invariant). This approach goes back to the original work of Gromov and Thurston \cite{GT}. The major challenge is to determine whether $k\to ||M_k||$ is a strictly increasing function. While this is open as far as we can tell, it is possible to relax the requirement and prove that this function has suitable monotonicity properties. This is what we discuss next.

In the spirit of \cite{GT}, one can argue as follows. Fix $\ep\in(0,1)$. Provided that the normal injectivity radius of $B$ in $M$ is large enough, every $M_k$ can be equipped with a Riemannian metric $g_k$ of $(-1-\ep,-1+\ep)$-pinched negative curvature. Furthermore, the volume of such metric satisfies ${\rm vol}(M_k,g_k)/k\cdot{\rm vol}(M)\in(1/c_1,c_1)$ for some uniform constant $c_1>1$ only depending on $M$. By work of Besson, Courtois, and Gallot \cite{BCG} for every $\ep\in(0,1)$ there exists $c_2=c_2(d,\ep)>0$ such that if $X$ and $Y$ are homotopy equivalent closed Riemannian $d$-manifolds with pinched negative sectional curvature contained in the interval $(-1-\ep,-1+\ep)$ then ${\rm vol}(X)/{\rm vol}(Y)\in(1/c_2,c_2)$. Thus, in our setup, if $M_k$ is homotopy equivalent to $M_{k'}$ then we must have $1/c_1c_2\le k/k'\le c_1c_2$. This provides a coarse way to distinguish Gromov-Thurston manifolds. Note that, thanks to the Proportionality Principle \cite{Gromov}, in pinched negative curvature the ratio between the Riemannian volume and the simplicial volume is uniformly bounded from above and below. 

In the last part of the introduction, we describe the organisation of the article and introduce the main ingredients of our strategy.

\subsection*{Overview and outline}

Section \ref{sec:geom_branch} contains various facts known to experts about Gromov-Thurston manifolds that can be deduced using their ${\rm CAT}(-1)$-geometry.

In Section \ref{sec:vir_Dehn_filling} we describe fundamental groups of Gromov-Thurston manifolds for a GT-pair $(M,B)$ as virtual Dehn fillings of $\pi_1(M-B)$ (see Lemma \ref{lem:van kampen} and Lemma \ref{lem:rotation is conjugation} for precise statements). This is crucial to our approach, and ultimately it is an application of van Kampen's theorem. The group $\pi_1(M-B)$ is relatively hyperbolic by \cite{Bel}. In Section \ref{sec:rel hyp}, we describe some important consequences of this fact in relation to virtual Dehn fillings. 

Section \ref{sec:actions_on_trees} is the core of this paper. The main result is Theorem \ref{thm:displ}, which is inspired by \cite{DG18}. The theorem allows us to bound the number of injective homomorphisms from hyperbolic groups (with suitable non-splitting properties) into the fundamental groups of Gromov-Thurston manifolds. The strategy of the proof is to proceed by contradiction, starting with an unbounded (in a suitable sense) sequence of homomorphisms and extracting a limit action on an $\mathbb R$-tree (using Bestvina and Paulin's technology, see \cite{Paulin}). We then have to study the arc stabilisers for this action, and these can be more complicated than in \cite{DG18}. This creates additional obstacles to the application of Rips theory that we have to address by performing additional constructions. This is also why we need to assume that the hyperbolic groups that we consider have stronger non-splitting properties than in \cite{DG18}. We find three different possible scenarios, and for each of them we use a different method to construct an action on a suitable tree. The three methods we use are: The theory of tree-graded spaces \cite{DS}, combinatorial considerations involving the arrangement of limits of horoballs, and taking a limit of coned-off graphs. The non-splitting assumption will rule out all these actions and produce the desired contradiction.

In Section \ref{sec:no_split} we show that Theorem \ref{thm:displ} applies to fundamental groups of complements of walls, by showing that they do not have certain types of splittings. This involves a mix of homological and geometric arguments.

In Section \ref{sec:proof main} we study outer automorphism groups, relying on Theorem \ref{thm:displ}, and complete the proofs of Theorems \ref{thmA:main}, \ref{thmA:Out}, and \ref{thmA:Sylow}.

\subsection*{Acknowledgements}
We warmly thank Ursula Hamenstädt for useful discussions. We also thank Giorgio Mangioni, Jason Manning, and Bruno Premoselli for their feedback on previous drafts of this article.

\section{Geometry and topology of the branched cover}
\label{sec:geom_branch}

Let $M$ be a closed orientable hyperbolic $d$-manifold containing a connected totally geodesic embedded codimension 2 submanifold $B\subset M$. Suppose further that $M$ is the boundary $B=\partial W$ of a codimension 1 totally geodesic orientable submanifold $W\subset M$.

\begin{dfn}[Wall, Sector, and Singularity]
In the setting above, we call $W$ the {\em wall}, $S:=M-W$ the (incomplete) {\em sector}, and $B$ the {\em singularity}.    
\end{dfn}

\begin{dfn}[Complete Sector]
\label{dfn:sector}
The abstract completion $\bar{S}$ of the hyperbolic manifold $S:=M-W$ with respect to the intrinsic path metric is a compact concave $n$-manifold with totally geodesic boundary pleated along the singularity (we call it the {\em complete sector}). More precisely, $\partial\bar{S}$ consists of two copies $W^+,W^-$ of $W$ glued together along their common boundary $B$ where they form an angle of $2\pi$. Endowed with the intrinsic path metric, the boundary $\partial\bar{S}$ is isometric to the double $DW$ of the wall $W$ along $\partial W=B$.    
\end{dfn}

Associated with $M,W,B$, there is a canonical degree $k$ branched cover $M_k\to M$.

\begin{dfn}[Branched Cover]
\label{dfn:branched cover}
Consider $k$ copies $\bar{S}_1,\cdots,\bar{S}_k$ (indexed by $j\in\mb{Z}/k\mb{Z}$) of a complete sector $\bar{S}$. The {\em degree $k$ branched cover $M_k$ associated with $M,W,B$} is
\[
M_k:=\bigsqcup_{j\in\mb{Z}/k\mb{Z}}{\bar{S}_j}\left/\,\{W_j^+=W_{j+1}^-:j\in\mb{Z}/k\mb{Z}\}\right.
\]
where we are gluing the copies of the complete sector along the canonical cyclic identifications of their boundaries. In particular, note that the intersection of all $\bar{S}_j\subset M_k$ is a copy of $B$, which we denote by $B_k\subset M_k$. The manifold $M_k$ comes equipped with 
\begin{itemize}
    \item{A projection $p_k:M_k\to M$ mapping each $S_j\subset\bar{S}_j$ isometrically to $S=M-W$. The map $p_k$ is a degree $k$ branched covering branched over $B\subset M$.}
    \item{An order-$k$ diffeomorphism $\rho_k:M_k\to M_k$ that fixes $B_k$ and maps $\bar{S}_j$ to $\bar{S}_{j+1}$ isometrically for each $j\in\mb{Z}/k\mb{Z}$. We call $\rho_k$ the {\em canonical rotation} of $M_k$.}
\end{itemize}
\end{dfn}

The goal of this section is to exploit the path metric induced by the natural hyperbolic cone structure on $M_k$ to get useful geometric information about the fundamental groups of $W,S,M_k$ and how they sit one inside the other. 

\subsection{Geometry}
First of all, let us recall that the standard $d$-dimensional hyperbolic cone of angle $\theta$ is the singular Riemannian $d$-manifold 
\[
\mb{H}^d(\theta)=[0,\infty)\times\mb{S}^1_\theta\times\mb{H}^{d-2}/\{(0,\phi,x)\sim(0,\phi',x):x\in\mb{H}^{d-2},\phi,\phi'\in\mb{S}^1_\theta\},
\]
where $\mb{S}^1_\theta=\mb{R}/\theta\mb{Z}$, equipped with the singular Riemannian metric 
\[
{\rm d}r^2+\sinh(r)^2{\rm d}\theta^2+\cosh(r)^2{\rm d}\sigma_{\mb{H}^{d-2}}^2
\]
where $\sigma_{\mb{H}^{d-2}}$ is the hyperbolic metric on $\mb{H}^{d-2}$. We endow $\mb{H}^d(\theta)$ with the natural path metric induced by the singular Riemannian metric. The following is a well-known fact.

\begin{pro}
\label{pro:cat-1cone}
Assume $\theta\ge2\pi$. Then $\mb{H}^d(\theta)$ is ${\rm CAT}(-1)$.
\end{pro}

We also want to consider sectors inside $\mb{H}^d(\theta)$. Denote by $\phi:\mb{H}^d(\theta)-\mb{H}^{d-2}\to\mb{S}^1_\theta$ the projection to the circle. For every $z\in\mb{H}^d(\theta)-\mb{H}^{d-2}$ define $H_z\subset\mb{H}^d(\theta)$ to be the half-hyperplane bounding the singularity $\mb{H}^{d-2}$ and containing $z$, in other words, 
\[
H_z=\{w\in\mb{H}^{d-2}(\theta):\phi(w)=\phi(z)\}\cup\mb{H}^{d-2}.
\]
Two (ordered) half-hyperplanes $H_x,H_y$ as above bound a sector 
\[
V_{x,y}=\{w\in\mb{H}^d(\theta)-\mb{H}^{d-2}:\phi(x)\le\phi(w)\le\phi(y)\}\cup\mb{H}^{d-2}
\]
where we orient $\mb{S}^1_\theta$ counterclockwise and we denote by $\phi\le\psi\le\phi'$ the arc between $\phi,\phi'$ according to this orientation. Note that the isometry type of $V$ only depends on the angle $\phi(y)-\phi(x)$. We denote by $V(\psi)$ the sector with angle $\psi\in[0,\theta)$. The following is also a well-known fact.

\begin{lem}
\label{lem: tot geo}    
Assume $\theta>2\pi$. Let $x,y\in\mb{H}^d(\theta)-\mb{H}^{d-2}$ be non-singular points. 
\begin{enumerate}
    \item{If the two angles formed by $H_x,H_y$ at $\mb{H}^{d-2}$ are at least $\pi$ then inclusion $H_x\cup H_y\to\mb{H}^d(\theta)$ is isometric.}
    \item{If the angle of $V_{x,y}$ is at most $\theta-\pi$ then the sector $V_{x,y}$ is convex. In particular, the restriction of the path metric to $V_{x,y}$ is ${\rm CAT}(-1)$}.
\end{enumerate}
\end{lem}

We give a proof sketch.

\begin{proof}
Note that Property (2) easily follows from Property (1), which guarantees that the boundary of the sector is convex.

Property (1) follows from an explicit description of the geodesics in $\mb{H}^d(\theta)$. 

Let us start with a preliminary observation. Each geodesic of $\mb{H}^d(\theta)$ contained in $\mb{H}^d(\theta)-\mb{H}^{d-2}$ coincides with a Riemannian geodesic for the Riemannian metric on $\mb{H}^d(\theta)-\mb{H}^{d-2}$ because it is length minimising.

Consider a non-singular point $z\in \mb{H}^d(\theta)-\mb{H}^{d-2}$. Define $O_z$ to be the open region containing $z$ and bounded by the two half-hyperplanes $H_z',H_z''$ forming an angle of $\pi$ with $H_z$ at the singularity $\mb{H}^{d-2}$. 

Note that $O_z$ equipped with the intrinsic Riemannian metric admits a Riemannian isometry $\phi:O_z\to\mb{H}^d-H$ where $H$ is a half-hyperplane contained in a hyperplane $U\subset \mb{H}^d$ with $\phi(z)\in U-H$. Furthermore, $\phi$ extends continuously to $\phi:O_z\cup\partial O_z\to\mb{H}^d$ and the restrictions of $\phi$ to $H_z',H_z''$ are both isometries onto $H$. 

The key observation is that every geodesic $\tau:[0,r]\to\mb{H}^d(\theta)$ starting at $z$ and intersecting $\partial O_z=H_z'\cup H_z''$ must be contained in $H_z$ and intersect the singularity $\mb{H}^{d-2}=H_z'\cap H_z''$. Indeed, let $t\in(0,r)$ be the first time $\tau$ hits $\partial O_z$, say $\tau(t)\in H_z'$. The geodesic $\tau[0,t]$ is a Riemannian geodesic for the intrinsic Riemannian metric on $O_z$. Thus $\phi\tau[0,t]$ is a Riemannian geodesic for $\mb{H}^d-H$ starting at $\phi(z)\in U-H$ and limiting to a point in $\phi(\tau(t))\in H$. By the structure of the geodesics in $\mb{H}^d$, being both $\phi(z)$ and $H$ contained in the hyperplane $U$, the only possibility is that the limit point is on the singularity $\partial H$ and, hence, $\tau(t)\in\mb{H}^{d-2}$. 

We are now able to conclude the proof of Property (1). Consider two arbitrary non-singular points $x,y\in\mb{H}^d(\theta)-\mb{H}^{d-2}$ such that both angles formed by $H_x,H_y$ at $\mb{H}^{d-2}$ are at least $\pi$. In particular $y\not\in O_x$ and $x\not\in O_y$. Let $\tau:[0,r]\to\mb{H}^d(\theta)$ be a geodesic joining $x$ to $y$. By the above discussion, an initial (resp. terminal) segment of $\tau$, say $\tau[0,t_x]$ (resp. $\tau[t_y,r]$), joins $x$ (resp. $y$) to the singularity $\mb{H}^{d-2}$ and is contained in $H_x$ (resp. $H_y$). The middle segment $\tau[t_x,t_y]$ joins two points of the singularity. As there is a natural 1-Lipschitz retraction $\mb{H}^d(\theta)\to\mb{H}^{d-2}$ which is strictly 1-Lipschitz outside $\mb{H}^{d-2}$, we have that $\tau[t_x,t_y]$, being length minimising, must be be contained in $\mb{H}^{d-2}$.

The conclusion follows.
\end{proof}

Let us go back to the branched cover $M_k$ and the sector $\bar{S}_j\subset M_k$. By construction, they are locally modeled on $\mb{H}^d(2k\pi)$ and $V(2\pi)\subset\mb{H}^d(2k\pi)$, we deduce the following information from Proposition \ref{pro:cat-1cone} and Lemma \ref{lem: tot geo}.

\begin{cor}
\label{cor:gromov hyperbolic}
The following holds.
\begin{itemize}
    \item{Both $\bar{S}$ and $M_k$ are locally ${\rm CAT}(-1)$.}
    \item{Both $\pi_1(S)$ and $\pi_1(M_k)$ are Gromov hyperbolic and torsion-free.}
    \item{The inclusion $\partial\bar{S}\subset\bar{S}$ is locally isometric.}
    \item{The inclusion $\bar{S}\subset M_k$ is locally isometric.}
\end{itemize}     
\end{cor}

\subsection{Topology}
As an application of the above discussion, we establish various $\pi_1$-injectivity properties.

\begin{lemma}
\label{lem:pi1injective}
The following holds.
\begin{enumerate}
    \item{The inclusion $B\to\bar{S}_j$ is $\pi_1$-injective.}
    \item{The inclusion $W\to\partial\bar{S}_j$ is $\pi_1$-injective.}
    \item{If $k\ge 2$ then the inclusion $\partial\bar{S}_j\to M_k$ is $\pi_1$-injective.}
    \item{If $k\ge 2$ then the inclusion $\bar{S}_j\to M_k$ is $\pi_1$-injective.}
\end{enumerate}
\end{lemma}

\begin{proof}
The four inclusions are all locally isometric (Corollary \ref{cor:gromov hyperbolic}). The conclusion follows from \cite[Proposition 4.14]{BH}.
\end{proof}

\section{Branched cover as a virtual Dehn filling}
\label{sec:vir_Dehn_filling}

Let $(M,B)$ be a {\rm GT-pair} with wall $W$. Let $M_k$ be the canonical degree $k$ branched cover associated with it. In this section we describe $\pi_1(M_k)$ as a virtual Dehn filling of $\pi_1(M-B)$. In particular, we will realise it as a canonical finite index subgroup of 
\[
\pi_1(M_k)<Q_k:=\pi_1(M-B)/\langle\langle\gamma^k\rangle\rangle
\]
where $\gamma$ is a generator of ${\rm ker}(\pi_1(U-B)\to \pi_1(U))$ (a so-called {\em meridian}) and $U$ is a tubular neighborhood of $B$ in $M$.

The main reason for relating $\pi_1(M_k)$ to $\pi_1(M-B)$ is that $M-B$ admits a complete locally ${\rm CAT}(-k)$-Riemannian metric $g$ for which a drilled neighborhood $U-B$ of $B$ lifts to a horoball in the universal cover of $(M-B,g)$. As a consequence, $\pi_1(M-B)$ is relatively hyperbolic with respect to $\pi_1(U-B)$ as first proved by Belegradek \cite{Bel} (see Theorem \ref{thm:rel hyp}). 

By the theory of algebraic Dehn fillings initiated in \cite{Osin,GM}, the groups $\pi_1(M_k)$ act on the (uniformly hyperbolic) spaces $X_k:=X/\langle\langle\gamma^k\rangle\rangle$ where $X$ is the cusp space for the relatively hyperbolic pair $(\pi_1(M-B),\pi_1(U-B))$. This fact plays a crucial role in the proof of the key structural results carried out in the next section.  

Another consequence of this point of view is that we can transform the automorphism of $\pi_1(M_k)$ provided by the order $k$ rotation $\rho_k:M_k\to M_k$ into a conjugation inside the group $Q_k$. This will facilitate a lot the study of the algebraic structure of $\pi_1(M_k)$ and of ${\rm Out}(\pi_1(M_k))$.

\subsection{Fundamental group}
The following lemma is crucial for our approach. It describes the fundamental groups of a Gromov-Thurston manifold $M_k$ as a virtual Dehn filling of $\pi_1(M-B)$. The proof requires careful bookkeeping of various identifications and isomorphisms, but it is ultimately an application of van Kampen's theorem.

\begin{lemma}
\label{lem:van kampen}
Let $U\subset M$ be a tubular neighborhood of $B$ in $M$ and let $y\in U-B$ be a basepoint. Denote by $U_k\subset M_k$ the pre-image of $U$ under the branched covering map $p_k:M_k\to M$ and choose a basepoint $y_k\in U_k-B_k$ such that $p_k(y_k)=y$. Consider $\gamma\in\pi_1(U-B,y)$ a generator of ${\rm ker}(\pi_1(U-B,y)\to\pi_1(U,y))$. The projection $p_k:M_k-B_k\to M-B$ induces an injective homomorphism
\[
\pi_1(M_k,y_k)\to Q_k:=\pi_1(M-B,y)/\langle\langle \gamma^k\rangle\rangle
\]
whose image is an index-$k$ normal subgroup. The quotient group is cyclic of order $k$. 
\end{lemma}

\begin{proof}
Recall that the manifold $M_k$ is constructed as
\[
M_k:=\bigsqcup_{j\in\mb{Z}/k\mb{Z}}{\bar{S}_j}\left/\,\{W_j^+=W_{j+1}^-:j\in\mb{Z}/k\mb{Z}\}\right..
\]

We can assume that $U\subset M$ intersects the wall $W$ in a collar neighborhood $V\subset W$ of $B=\partial W$ in $W$. In particular, $U$ defines a neighborhood of $B$, denoted by $\bar{U}$, in the complete sector $\bar{S}$ intersecting the boundary $\partial\bar{S}=DW$ in the double of $V$, that is $\bar{U}\cap\partial\bar{S}=DV$. Denote by $V^+\subset W^+,V^-\subset W^-$ the two halves of $DV$ in $W^+,W^-$ respectively. The pre-image of $U$ in $M_k$ under the canonical branched cover projection $p_k:M_k\to M$ can be described as a gluing of $k$ copies $\bar{U}_j\subset\bar{S}_j$ of $\bar{U}$ with identifications $V_j^+=V_{j+1}^-$, that is
\[
U_k:=\bigsqcup_{j\in\mb{Z}/k\mb{Z}}{\bar{U}_j}\left/\,\{V_j^+=V_{j+1}^-:j\in\mb{Z}/k\mb{Z}\}\right..
\]

From the above discussion it follows that $p_k:U_k\to U$ is $\mb{Z}/k\mb{Z}$-equivariantly equivalent to the standard branched cover $({\rm id},\pi_k):B\times\mb{D}^2\to B\times\mb{D}^2$ where $\mb{D}^2$ is the open unit disk and $\pi_k:\mb{D}^2\to\mb{D}^2$ is map sending (in polar coordinates) $\pi_k(r,\theta)=(r,k\theta)$ (it is a branched cover of $\mb{D}^2$ of degree $k$ branched over $0\in\mb{D}^2$). In particular $p_k$ maps a generator $\gamma_k$ of the kernel ${\rm ker}(\pi_1(U_k-B_k,y_k)\to\pi_1(U_k,y_k))$ (represented by a circle centered at $0$ in $\{{\rm pt}\}\times\mb{D}^2$) to $p_k(\gamma_k)=\gamma^k$ where $\gamma$ is a generator of ${\rm ker}(\pi_1(U-B,y)\to\pi_1(U,y))$ (represented by a circle centered at $0$ in $\{{\rm pt}\}\times\mb{D}^2$).

By van Kampen, we have
\[
\pi_1(M_k,y_k)\simeq\pi_1(M_k-B_k,y_k)*_{\pi_1(U_k-B_k,y_k)}\pi_1(U_k,y_k).
\]

By the above discussion, $U_k$ is a tubular neighborhood of $B_k$ in $M_k$ and we have that $\pi_1(U_k-B_k,y_k)\to\pi_1(U_k,y_k)$ is surjective. Thus, we have a canonical identification
\[
\pi_1(M_k-B_k,y_k)*_{\pi_1(U_k-B_k,y_k)}\pi_1(U_k,y_k)\simeq\pi_1(M_k-B_k,y_k)/\langle\langle \gamma_k\rangle\rangle.
\]
where ${\rm Ker}(\pi_1(U_k-B_k,y_k)\to\pi_1(U_k,y_k))=\langle\gamma_k\rangle$. By the above discussion, we also have $p_k(\gamma_k)=\gamma^k$. In particular, the map $(p_k)_*:\pi_1(M_k-B_k,y_k)\to Q_k=\pi_1(M-B,y)/\langle\langle\gamma^k\rangle\rangle$ passes to the quotient $\pi_1(M_k-B_k,y_k)/\langle\langle\gamma_k\rangle\rangle$ and, by the van Kampen identification, it induces a map
\[
(p_k)_*:\pi_1(M_k,y_k)\simeq\pi_1(M_k-B_k,y_k)/\langle\langle\gamma_k\rangle\rangle\to Q_k.
\]

Finally, let us discuss injectivity, normality of the image, and the quotient group. Essentially, everything boils down to standard properties of covering spaces. Observe that $p_k:M_k-B_k\to M-B$ is a covering map of degree $k$ and that the rotation $\rho_k:M_k\to M_k$ restricts to an order $k$ covering automorphism acting transitively on the fibers of $p_k:M_k-B_k\to M-B$. Thus, the latter is a regular covering and its automorphism group is the cyclic group $\langle\rho_k\rangle$. By covering theory, this implies that $(p_k)_*:\pi_1(M_k-B_k,y_k)\to\pi_1(M-B,y)$ is injective and that the image is an index-$k$ normal subgroup (containing $\langle\langle\gamma^k\rangle\rangle$ as $(p_k)_*(\gamma_k)=\gamma^k$) whose quotient group is isomorphic to $\langle\rho_k\rangle\simeq\mb{Z}/k\mb{Z}$.

As a consequence, the image of $(p_k)_*:\pi_1(M_k-B_k,y_k)/\langle\langle\gamma_k\rangle\rangle\to Q_k$ is an index-$k$ normal subgroup with quotient isomorphic to $\mb{Z}/k\mb{Z}$. As $(p_k)_*:\pi_1(M_k-B_k,y_k)\to \pi_1(M-B,y)$ is injective and the image is normal and contains $(p_k)_*(\gamma_k)=\gamma^k$, the intersection between the image group and the subgroup $\langle\langle\gamma^k\rangle\rangle$ is exactly the image of the normal subgroup $\langle\langle\gamma_k\rangle\rangle<\pi_1(M_k-B_k,y_k)$. It follows that the map $(p_k)_*:\pi_1(M_k-B_k,y_k)/\langle\langle\gamma_k\rangle\rangle\to Q_k$ is injective. 
\end{proof}

\subsection{The rotation is a conjugation}
Having the description of $\pi_1(M_k)$ as a virtual Dehn filling of $\pi_1(M-B)$ provided by Lemma \ref{lem:van kampen}, we turn to the action of the rotation $\rho_k:M_k\to M_k$ (see Definition \ref{dfn:branched cover}) on $\pi_1(M_k)$. For convenience, we conflate $\rho_k$ with the automorphism of $\pi_1(M_k)$ that it induces. The next lemma identifies this automorphism with a conjugation in $Q_k$.

\begin{lemma}
\label{lem:rotation is conjugation}
Notations as in Lemma \ref{lem:van kampen}. Consider the injective homomorphism 
\[
\pi_1(M_k,y_k)\to Q_k=\pi_1(M-B,y)/\langle\langle\gamma^k\rangle\rangle
\]
induced by the branched covering projection $p_k:M_k\to M$ (described in Lemma \ref{lem:van kampen}). There is a representative of $\rho_k$ in ${\rm Out}(\pi_1(M_k,y_k))$ whose action on the image of $\pi_1(M_k,y_k)$ coincides with the restriction to the image of $\pi_1(M_k,y_k)$ of the conjugation by $\gamma$ in $Q_k$. Furthermore, the projection of $\gamma$ to the order-$k$ cyclic group $Q_k/{\rm Im}(\pi_1(M_k,y_k)\to Q_k)$ (by Lemma \ref{lem:van kampen}) is a generator.
\end{lemma}

\begin{proof}
As in Lemma \ref{lem:van kampen}, the branched cover $p_k:U_k\to U$ is $\mb{Z}/k\mb{Z}$-equivariantly equivalent to the standard branched cover $({\rm id},\pi_k):B\times\mb{D}^2\to B\times\mb{D}^2$ where $\pi_k:\mb{D}^2\to\mb{D}^2$ is the map that in polar coordinates behaves as $\pi_k(r,\theta)=(r,k\theta)$. Under this equivalence $y\in U-B$ corresponds to $({\rm pt},re^{i\theta})\in B\times\mb{D}^2$. The radial path in $B\times\mb{D}^2$ given by $t\in[0,1]\mapsto ({\rm pt},t\cdot re^{i\theta})$ corresponds to a path $\eta:[0,1]\to U_k$ joining $x_k\in B_k$ to $y_k\in U_k-B_k$.

The path $\eta$ induces an isomorphism $\phi_\eta:\pi_1(M_k,x_k)\to\pi_1(M_k,y_k)$ sending
\[
[\alpha]\to[\eta^{-1}\star\alpha\star\eta].
\]

The rotation $\rho_k$ naturally induces an automorphism $(\rho_k)_*:\pi_1(M_k,x_k)\to\pi_1(M_k,x_k)$ and via the above isomorphism $\phi_\eta$ it induces also an automorphism of $\pi_1(M_k,y_k)$ which we now describe.

Let $\delta$ be the path in $U_k-B_k$ obtained by homotoping the concatenation $\eta^{-1}\star\rho_k(\eta)$ so that it misses $B_k$ in the following way. Working inside the slice of $U_k$ corresponding to $\{{\rm pt}\}\times\mb{D}^2$, the concatenation $\eta\star\rho_k(\eta)$ corresponds to the concatenation of the radial paths $t\mapsto({\rm pt},t\cdot re^{i\theta})$ and $t\mapsto({\rm pt},(1-t)\cdot re^{i(\theta+2\pi/k)})$. We can then homotope it to a circular arc $t\to re^{i(\theta+t2\pi/k)}$ joining $re^{i\theta}$ to $re^{i(\theta+2\pi/k)}$. The path $\delta$ is the path in $U_k$ that corresponds to this circular arc. Note that, since the $\pi_k$-projection of the circular arc to $\mb{D}^2$ is a full circle, the $p_k$-projection of $\delta$ is a circle in $U$ around $B$ which is a generator $\gamma$ of ${\rm ker}(\pi_1(U-B,y)\to\pi_1(U,y))$.

In conclusion, we have
\[
\phi_\eta(\rho_k)_*\phi_\eta^{-1}[\alpha]=[\delta\star\rho_k(\alpha)\star\delta^{-1}].
\]

By Lemma \ref{lem:van kampen}, the image in $Q_k=\pi_1(M-B,y)/\langle\langle\gamma^k\rangle\rangle$ of $\phi_\eta(\rho_k)_*\phi_\eta^{-1}[\alpha]$ is 
\[
(p_k)_*(\phi_\eta(\rho_k)_*\phi_\eta^{-1}[\alpha])=[p_k(\delta)\star p_k(\rho_k(\alpha))\star p_k(\delta)^{-1}].
\]

Since $p_k\rho_k=p_k$ we have that $p_k(\rho_k(\alpha))=p_k(\alpha)$. Recalling that $p_k(\delta)=\gamma$, the desired identification between the action of $\rho_k$ and the conjugation by $\gamma$ follows.

It remains to check that the projection of $\gamma$ to the cyclic group $Q_k/{\rm Im}(\pi_1(M_k,y_k)\to Q_k)$ is a generator. Recall that $Q_k/{\rm Im}(\pi_1(M_k,y_k)\to Q_k)$ is isomorphic to $\pi_1(M-B,y)/{\rm Im}(\pi_1(M_k-B_k,y_k)\to \pi_1(M-B,y))$ and that the latter is identified by standard covering theory with the deck group of $p_k:M_k-B_k\to M-B$. As already observed in the proof of Lemma \ref{lem:van kampen}, the deck group is the cyclic group $\langle\rho_k\rangle$ generated by the (restriction of the) rotation $\rho_k:M_k-B_k\to M_k-B_k$. Hence, in order to check that $\gamma$ is a generator, it is enough to check that the automorphism of $p_k:M_k-B_k\to M-B$ corresponding to $[\gamma]\in\pi_1(M-B,y)/{\rm Im}(\pi_1(M_k-B_k,y_k)\to \pi_1(M-B,y))$ is a generator of the deck group. By covering theory, the deck transformation corresponding to $[\gamma]$ is the unique one that maps the initial point of a lift of $\gamma$ to $M_k-B_k$ to the terminal point of the same lift. In our case, we can choose as a lift of $\gamma$ the path $\delta$ described before. Its terminal point is exactly the rotation $\rho_k$ applied to its initial point. Thus $[\gamma]$ corresponds to $\rho_k$ as desired. 
\end{proof}

\subsection{Relative hyperbolicity}
\label{sec:rel hyp}
We now move to the geometric part of the section.

The following relative hyperbolicity result is a key ingredient.

\begin{thm}[{\cite{Bel}}]
\label{thm:rel hyp}
Let $M$ be a closed hyperbolic $d$-manifold and $B\subset M$ a codimension 2 totally geodesic submanifold. Denote by $U$ a tubular neighborhood of $B$ in $M$. Then the pair $(\pi_1(M-B),\pi_1(U-B))$ is relatively hyperbolic.
\end{thm}

\begin{dfn}[Cusped Space]
We denote by $X$ the cusped space for the pair $(\pi_1(M-B),\pi_1(U-B))$. We have that  $\pi_1(M_k)<Q_k=\pi_1(M-B)/\langle\langle\gamma^k\rangle\rangle$ acts on the quotient $X_k:=X/\langle\langle\gamma^k\rangle\rangle$. We denote by $P_k$ the quotient of $\pi_1(U-B)<\pi_1(M-B)$.    
\end{dfn}

The following is a crucial property. It states that $X_k$ is a $\delta$-hyperbolic geodesic metric space where $\delta$ is a uniform constant independent of $k$.

\begin{pro}[{\cite[Proposition 2.3]{AGM09}}]
\label{pro:unif hyp}
There exist $\delta,k_0>0$ such that $X_k$ is $\delta$-hyperbolic for all $k\geq k_0$.
\end{pro}

In particular, we get a collection of actions $\pi_1(M_k)\curvearrowright X_k$ on uniformly hyperbolic geodesic metric spaces. In the next section, we will study the possible degenerations of such actions by analysing their natural geometric limits.

Towards this goal, let us recall a defining property of the cusped space $X$ for the relative hyperbolic pair $(\pi_1(M-B),\pi_1(U-B))$ and how it is reflected in its quotients $X_k=X/\langle\langle\gamma^k\rangle\rangle$. 

The Gromov hyperbolic space $X$ contains a collection $\mc{H}$ of pairwise disjoint horoballs $H\subset X$ which is precisely invariant under $\pi_1(M-B)$, that is, such that for every $H\in\mc{H}$ and $\alpha\in\pi_1(M-B)$ either $\alpha H\cap H=\emptyset$ or $\alpha H=H$. Furthermore, there is a bijective correspondence between the horoballs $H\in\mc{H}$ and the conjugates of $\pi_1(U-B)$, namely, each conjugate of $\pi_1(U-B)$ stabilises some $H\in\mc{H}$ (on which it acts by parabolic isometries fixing the center at infinity of the horoball) and for every $H\in\mc{H}$ the stabiliser ${\rm Stab}(H)$ is a conjugate of $\pi_1(U-B)$. Once the horoballs in $\mc{H}$ are removed, the action becomes cocompact, that is, 
\[
\left(X-\bigcup_{H\in\mc{H}}{H}\right)/\pi_1(M-B)
\]
is compact.
    
Denote by $\pi_k:X\to X_k=X/\langle\langle\gamma^k\rangle\rangle$ the quotient projection. 

The collection $\mc{H}$ descends to $\mc{H}_k:=\{H_k:=\pi_k(H)\,|\,H\in\mc{H}\}$. It is immediate to verify that the family $\mc{H}_k$ is precisely invariant under the group $Q_k=\pi_1(M-B)/\langle\langle\gamma^k\rangle\rangle$ naturally acting on $X_k$. Since $\pi_1(M-B)$ is relatively hyperbolic with respect to $\pi_1(U-B)$ and $\gamma$ has infinite order, it is well-known \cite{Osin, GM} that for all sufficiently large $k$ the intersection $\langle\langle\gamma^k\rangle\rangle\cap\pi_1(U-B)$ coincides with the normal subgroup generated by $\gamma^k$ in $\pi_1(U-B)$. The latter agrees with the subgroup $\langle\gamma^k\rangle<\pi_1(U-B)$ as $\gamma$ generates the normal subgroup ${\rm ker}(\pi_1(U-B)\to\pi_1(U))$. Thus, the projection of $\pi_1(U-B)$ to $Q_k$ induces an injection
\[
\pi_1(U-B)/\langle\gamma^k\rangle\to Q_k.
\]
Note that since $\pi_1(U)=\pi_1(U-B)/\langle\gamma\rangle$, the group $P_k:=\pi_1(U-B)/\langle\gamma^k\rangle$ is a finite extension of $\pi_1(U)=\pi_1(B)$. More precisely, we have $\pi_1(U-B)\simeq\pi_1(B)\times\mb{Z}\gamma$ and $P_k\simeq\pi_1(B)\times\mb{Z}/k\mb{Z}\gamma$. Equipped with this notation, it is a routine argument to check that the stabilisers of $H_k\in\mc{H}_k$ correspond precisely to the conjugates in $Q_k$ of $P_k$.

\subsection{Torsion and centre}
Torsion elements and centralisers in Dehn filling of relatively hyperbolic groups are well-understood. We collect here some useful characterisations that will be very useful for us in later sections.

\begin{pro}[{\cite{Osin,GM}}]
\label{pro:order_k}
For all sufficiently large $k$, the map 
\[
\pi_1(U-B)/\langle\langle\gamma^k\rangle\rangle=\pi_1(B)\times\mb{Z}/k\mb{Z}\gamma\to Q_k=\pi_1(M-B)/\langle\langle \gamma^k\rangle\rangle
\]
is injective. In particular, the image of $\gamma$ has order exactly $k$ in $Q_k$.
\end{pro}

\begin{lemma}{\cite[Lemma 4.3]{DG18}}
\label{lem:filling_finite_order}
For all sufficiently large $k$, an element $g$ of $\pi_1(M-B)/\langle\langle \gamma^k\rangle\rangle$ has finite order if and only if it is conjugate to a power of $\gamma$.
\end{lemma}

\begin{lemma}
\label{lem:centre}
For all sufficiently large $k$, any non-elementary subgroup of $Q_k=\pi_1(M-B)/\langle\langle \gamma^k\rangle\rangle$ not conjugate into the image $P_k$ of $\pi_1(U-B)$ has trivial centraliser. In particular, $\pi_1(M-B)/\langle\langle \gamma^k\rangle\rangle$ has trivial centre.
\end{lemma}

\begin{proof}
It is well-known that non-elementary subgroups of hyperbolic groups have finite centraliser. By Lemma \ref{lem:filling_finite_order}, finite-order elements in our case are parabolic. We now argue that all non-trivial finite-order parabolic elements of $\pi_1(M-B)/\langle\langle \gamma^k\rangle\rangle$ have exactly one fixed point in the Gromov boundary of $X_k$, and therefore they cannot centralise any subgroup not contained in the image of $\pi_1(U-B)$.
    
By Proposition \ref{pro:order_k}, we have to argue that $\gamma$ only has one fixed point in the boundary. This is because, for $k$ sufficiently large, $\gamma$ has a power that moves each point on the horosphere that it stabilises much more than the hyperbolicity constant of $X_k$ (which is uniform in $k$ by Proposition \ref{pro:unif hyp}). If $\gamma$ fixed more than one point in the boundary of $X_k$, then any of its powers would coarsely fix a geodesic line with one endpoint the limit point of the aforementioned horosphere. But this is not compatible with a power of $\gamma$ having large translation distance on the horosphere. 
\end{proof}

As a simple application of the above properties, we show here that the rotation $\rho_k$ has order exactly $k$ in ${\rm Out}(\pi_1(M_k))$ (not so clear a priori). Its proof is a prototypical example of how we combine properties of the Dehn filling $Q_k$ with the identification of the action of $\rho_k$ with a conjugation in $Q_k$. Similar arguments will appear in Section \ref{sec:proof main}.

\begin{cor}
\label{cor:order_k}
For all sufficiently large $k$, the outer automorphism represented by $\rho_k$ has order exactly $k$ in ${\rm Out}(\pi_1(M_k))$.
\end{cor}

\begin{proof}
For convenience, we conflate $\rho_k$ with the automorphism it induces at the level of $\pi_1$. By Lemma \ref{lem:rotation is conjugation}, $\rho_k$ coincides with the restriction of the conjugation by some element $\gamma\in Q_k$, when $\pi_1(M_k)$ is seen as a subgroup of it (as in Lemma \ref{lem:van kampen}).  If $\rho_k^j$ is trivial in ${\rm Out}(\pi_1(M_k))$, this means that there exists $g\in \pi_1(M_k)$ such that conjugation by $g$ coincides with conjugation by $\gamma^j$ on $\pi_1(M_k)$. That is, $g^{-1}\gamma^j$ is contained in the centraliser of $\pi_1(M_k)$ in $Q_k$, which is trivial by Lemma \ref{lem:centre}. Since $\gamma$ has order exactly $k$ (Proposition \ref{pro:order_k}), we conclude that this happens if and only if $j$ is a multiple of $k$, and therefore $\rho_k$ has order exactly $k$.
\end{proof}

\section{Injections into $\pi_1(M_k)$}
\label{sec:actions_on_trees}

Throughout this section, $(M,B)$ is a GT-pair with wall $W$, where $M$ has dimension at least 3, and $M_k$ is the canonical degree $k$ branched cover associated with $M,W,B$.

Recall that we denote $Q_k=\pi_1(M-B)/\langle\langle \gamma^k\rangle\rangle$, where $\gamma$ is a generator of ${\rm ker}(\pi_1(U-B)\to\pi_1(U))$ and $U$ is a tubular neighborhood of $B$ in $M$ (Lemma \ref{lem:van kampen}). Recall also that $\pi_1(M-B)$ is relatively hyperbolic with respect to $\pi_1(U-B)$ (Theorem \ref{thm:rel hyp}) and $Q_k$ acts on what we denote by $X_k$, namely the quotient of a cusped space $X$ for the pair $(\pi_1(M-B),\pi_1(U-B))$ by $\langle\langle\gamma^k\rangle\rangle$. Crucially, there exist $\delta,k_0$ such that $X_k$ is $\delta$-hyperbolic for every $k\ge k_0$ (Proposition \ref{pro:unif hyp}). 

In this section, we show, roughly, that hyperbolic groups which do not split with respect to a suitable family of subgroups cannot have arbitrarily wild embeddings into the $Q_k$. In order to state this precisely, we need two definitions.

\begin{dfn}[Minimal Displacement]
Let $G$ be a group, generated by a finite set $S$, acting by isometries on $(Z,d_Z)$. For $z\in Z$ we define the {\em displacement of $S$ at $z$} as 
\[
{\rm displ}_{S,Z}(z)=\max_{s\in S} d_Z(z,sz).
\]
The {\em minimal displacement of $S$} is 
\[
{\rm min-displ}(S,Z):=\inf_{z\in Z}{\rm displ}_{S,Z}(z).
\]
\end{dfn}

\begin{dfn}[Locally Covers]
\label{dfn:locally covers}    
We say that a group $H$ \emph{locally covers $B$} if all finitely generated subgroups of $H$ are isomorphic to subgroups of $\pi_1(B)$.
\end{dfn} 

We prove the following. In the statement we include trivial splittings, so that any group splits over itself.

\begin{thm}
\label{thm:displ}
Let $(M,B)$ be a GT-pair of dimension at least 3. Also, let $G$ be a torsion-free hyperbolic group that does not split over any subgroup which locally covers $B$, and let $S$ be a finite generating set for $G$. Then there exists $D\geq 0$ such that for all $k\geq 1$ we have 
\[
{\rm min-displ}(\phi(S),X_k)\leq D
\]
for all injective homomorphisms $\phi:G\to Q_k$.
\end{thm}

The following consequence of the theorem is what we will actually use in our applications. Recall that we have an order-$k$ rotation $\rho_k:M_k\to M_k$ (Definition \ref{dfn:branched cover}). In the statement below, we conflate $\rho_k$ with the automorphism of $\pi_1(M_k)$ that it induces.

\begin{cor}
\label{cor:injections}
Suppose that $G$ and $(M,B)$ are as in Theorem \ref{thm:displ}. Then there exists $N>0$ such that for every $k$ there are at most $N$ injective homomorphisms $\phi:G\to\pi_1(M_k)$ up to postcomposition by a conjugation and an element of $\langle\rho_k\rangle$.
\end{cor}

\begin{proof}
By Lemma \ref{lem:van kampen}, we can regard $\phi$ as an injective homomorphism $G\to Q_k$. Consider a point $z_k\in X_k$ realising the minimal displacement up to an error of 1.

Recall from Section \ref{sec:rel hyp} that $X$ has a precisely invariant collection $\mc{H}$ of horoballs $H\subset X$, and that it descends to a $Q_k$-precisely invariant collection $\mc{H}_k$ in $X_k$. The stabiliser ${\rm Stab}(H_k)$ of $H_k$ in $Q_k$ is conjugate to the projection $P_k$ of $\pi_1(U-B)<\pi_1(M-B)$ in $Q_k$ which is isomorphic to $\pi_1(B)\times\mb{Z}/k\mb{Z}$. Set 
\[
Y:=X-\bigcup_{H\in\mc{H}}{H}\quad\text{\rm and}\quad Y_k:=X_k-\bigcup_{H_k\in\mc{H}_k}{H_k}
\]
and recall that $Y/\pi_1(M-B)=Y_k/Q_k$ is compact. 

A point $z_k\in X_k$ realising the minimal displacement up to an error of 1 cannot lie deeper than $D+1$ into the quotient $H_k$ of a horoball $H\in\mc{H}$, for otherwise all elements of $\phi(s)$ with $s\in S$ would satisfy $\phi(s)H_k\cap H_k\neq\emptyset$ and, hence, $\phi(s)\in{\rm Stab}(H_k)$ (as $\mc{H}_k$ is $Q_k$-precisely invariant). In particular, we would have that $G$ is isomorphic to $\phi(G)<{\rm Stab}(H_k)$. This is not possible as we now explain. Since $G$ is torsion-free and ${\rm Stab}(H_k)$ is isomorphic to $\pi_1(B)\times\mb{Z}/k\mb{Z}$, the restriction of the first factor projection ${\rm Stab}(H_k)\to\pi_1(B)$ to $\phi(G)$ is injective. Hence, $G$ is isomorphic to a subgroup of $\pi_1(B)$. This violates the fact that $G$ does not locally cover $B$ (in fact, in this case, $G$ trivially splits over itself).
    
Therefore, a point $z_k\in X_k$ realising the minimal displacement up to an error of 1 lies in the $(D+1)$-neighborhood of $Y_k$. The proof can now be completed as in \cite{DG18}. 

Up to conjugation, we can assume that $d_{X_k}(z_k,o_k)\le A$ where $o_k\in Y_k$ is the projection of a fixed basepoint $o\in Y$ and $A$ is a uniform constant independent of $k$ (it can be taken to be the diameter of $Y/\pi_1(M-B)$). Fix a word norm $|\cdot|$ on $\pi_1(M-B)$. By Milnor-Švarc, there exists $C>0$ such that $d_Y(o,g o)\ge|g|/C-C$ for every $g\in\pi_1(M-B)$ (here we equip $Y$ with its intrinsic path metric).

Now, on the one hand, we have (again, $Y_k$ is equipped with its intrinsic path metric) 
\[
d_{Y_k}(o_k,\phi(s)o_k)\le d_{X_k}(z_k,\phi(s)z_k)+2d_{X_k}(o_k,z_k)\le D+1+2A.
\]
On the other hand, we also have 
\[
d_{Y_k}(o_k,\phi(s)o_k)=\inf_{g_S\in\pi_1(M-B)\,\text{\rm lift of }\phi(s)}\{d_Y(o_k,g_So)\}\ge\inf_{g_S\in\pi_1(M-B)\,\text{\rm lift of }\phi(s)}\{|g_S|/C-C\}.
\]
Thus, for every $s\in S$ there exists a lift $g_S\in\pi_1(M-B)$ of $\phi(s)$ with uniformly bounded word norm $|g_S|\le C(D+1+2A+C)$ (independent of $k$). This immediately implies that, up to composition with a conjugation in $Q_k$, there is a uniform bound (independent of $k$) on the number of injective homomorphisms $\phi:G\to\pi_1(M_k)$.

To conclude, we need one last observation to clarify how conjugations by elements in $Q_k$ can be decomposed in conjugations in $\pi_1(M_k)$ and rotations. Since $\pi_1(M_k)<Q_k$ is normal (Lemma \ref{lem:van kampen}) and $\pi_1(M_k),\gamma$ generate $Q_k$ (Lemma \ref{lem:rotation is conjugation}), every element $\alpha\in Q_k$ can be written as $\alpha=\gamma^rg$ with $g\in\pi_1(M_k)$. Thus, the conjugation $c_\alpha$ is the composition of the conjugation by $\gamma^r$ and the conjugation by $g\in\pi_1(M_k)$. By Lemma \ref{lem:rotation is conjugation}, the conjugation by $\gamma^r$ coincides with $\rho_k^r$ on $\pi_1(M_k)$.
\end{proof}

The proof of Theorem \ref{thm:displ} will occupy the rest of the section.

For convenience of the reader, we will break it up into small steps.

\subsection{Reduction to large $k$}
Let $G$ be a group as in the statement of Theorem \ref{thm:displ} and $S\subset G$ a finite generating set. 

First of all, we argue that for a fixed $k$, there exists some $D_k$ such that
\[
{\rm min-displ}(\phi(S),X_k)\leq D_k
\]
for all injective homomorphisms $\phi:G\to Q_k$. Thus, the same holds for any fixed collection of finitely many $k$. 

Let us explain the argument for a fixed $k$. Note that $Q_k$ is hyperbolic since by Lemma \ref{lem:van kampen} it contains the finite-index subgroup $\pi_1(M_k)$ which is hyperbolic by Corollary \ref{cor:gromov hyperbolic}. In particular, the action of $Q_k$ on any of its Cayley graphs is acylindrical. 

We apply \cite[Theorem 1.2]{fin_pres_MCG} (which gives a more precise version of \cite{Delzant} even in the case of hyperbolic groups) with input $\Gamma=G$ and $A=S$ and the action of $G$ on a fixed Cayley graph $\mc{G}$ of $Q_k$ induced by $\phi$. We obtain as output a point $x\in\mc{G}$ and an element $\theta\in{\rm Mod}(G)$ such that $d_{\mc{G}}(x,\phi(\theta(s))x)\le D_k$ for every $s\in S$ where $D_k>0$ is a constant only depending on $G,S$ and the action $Q_k\curvearrowright\mc{G}$. The group ${\rm Mod}(G)$ is a particular subgroup of ${\rm Aut}(G)$. We do not need to describe it precisely, as we only need to know that it coincides with the inner automorphism subgroup ${\rm Inn}(G)$ when $G$ does not admit splittings over infinite cyclic subgroups. By our assumptions on $G$, this is the case for us. Thus, we conclude that $d_{\mc{G}}(y,\phi(s)y)\le D_k$ for every $s\in S$ where $y=\phi(g)^{-1}x$ and $\theta$ is the conjugation by $g\in G$.

From the uniform bound on the displacement $d_{\mc{G}}(y,\phi(s)y)\le D_k$ for the action on the Cayley graph $\mc{G}$ of $Q_k$ it is not hard to deduce the analogous statement for the action on $X_k$, that is, $d_{X_k}(z,\phi(s)z)\le D'_k$ for every $s\in S$ (where $z\in X_k$ is some point and $D_k'$ is a constant only depending on $D_k$). This follows immediately from the existence of a coarsely Lipschitz $Q_k$-equivariant map $\mc{G}\to X_k$.

In view of the argument above, in the remainder of this proof, we can assume $k\geq k_0$, for $k_0$ as in Proposition \ref{pro:unif hyp}.

The proof is by contradiction. We assume that the hyperbolic group $G$ as in the statement of Theorem \ref{thm:displ} has a sequence of injective homomorphisms $\phi_n:G\to Q_{k(n)}$ for which the minimal displacements of $\phi_n(S)$ diverge as
\[
d_n={\rm min-displ}(\phi_n(S),X_{k(n)})\geq n.
\] 
We will show that such a group $G$ must admit a suitable splitting over a subgroup that locally covers $B$, thus violating our initial hypothesis.

\subsection{The limit $\mb{R}$-tree and its peripheral arcs}
    
In the first part, we closely follow the argument for \cite[Proposition 5.8]{DG18} and produce an action of $G$ on an $\mb{R}$-tree.

We choose $o_n\in X_{k(n)}$ realising the minimal displacement $d_n$ up to an error of 1.

Recall that the $X_{k(n)}$ are uniformly hyperbolic by Proposition \ref{pro:unif hyp}.

Under the assumption of unbounded minimal displacements $d_n\to\infty$, we can rescale their metrics by a factor $1/d_n$ and pass to an ultralimit (with respect to a fixed non-principal ultrafilter $\omega$ on $\mb{N}$) to get an action of $G$ on an $\mathbb R$-tree $(T,d_T)$ with no global fixed point
\[
G\curvearrowright\left(X_{k(n)},\frac{1}{d_n}d_{X_{k(n)}}\right)\longrightarrow G\curvearrowright\left(T,d_T\right).
\]

Up to passing to a subsequence, we have convergence in the equivariant Hausdorff-Gromov topology. This means the following.

\begin{dfn}[Equivariant Hausdorff-Gromov]
\label{dfn:hg conv}
A sequence of isometric actions $G\curvearrowright(X_n,d_{X_n})$ converges to an isometric action $G\curvearrowright(Z,d_Z)$ in the {\em equivariant Hausdorff-Gromov topology} if for every finite set $\{z_1,\cdots,z_k\}\subset Z$, finite subset $F\subset G$, and (small) parameter $\ep>0$ the following holds. For every $n$ large enough there are points $x_1^n,\cdots,x_k^n\in X_n$ such that for every $g\in F$ and every $i,j$ we have
\[
|d_T(x_i,gx_j)-d_{X_n}(x_i^n,gx_j^n)|\le\ep.
\]
\end{dfn}

So far, all this is exactly as in \cite{DG18}. (At the corresponding point of their proof, Dahmani-Guirardel argue that peripherals fix points in trees, but for us, the domain for $\phi$ is hyperbolic, so it has no peripherals.)

In order to analyse the action $G\curvearrowright T$, it is crucial to understand the arc stabilisers. To this purpose, we single out a distinguished class of arcs which will be quite useful in our analysis. 

\begin{dfn}[Limit Horoball]
\label{dfn:lim hor}
Recall that every $X_k$ comes with a distinguished family of horoballs $\mc{H}_k$ (see Section \ref{sec:rel hyp}). By the ultralimit construction, each sequence 
\[
\left(H_n\in \mc{H}_{k(n)}\right)_{n\in\mb{N}}\quad\text{\rm with}\quad\sup_{n\in\mb{N}}{d_{X_{k(n)}}(o_n,H_n)/d_n}<\infty
\]
gives rise to a subset $H$ of $T$ which we call a {\em limit horoball}. We denote by $\mathcal P$ the collection of all these limit horoballs.
\end{dfn}

\begin{dfn}[Peripheral Arc]
We call an arc $\alpha\subset T$ {\em peripheral} if it is contained in a limit horoball $H\in\mc{P}$ (as in Definition \ref{dfn:lim hor}).
\end{dfn}

Limit horoballs have the following key properties.

\begin{lemma}
\label{lem:hor_inters}
Limit horoballs are closed, convex, and any two distinct ones intersect in at most one point.
\end{lemma}

\begin{proof}
Each $H\in\mc{P}$ is automatically a complete subspace, hence closed. Convexity immediately follows from the fact that all $X_k$, for $k$ sufficiently large, are hyperbolic with the same hyperbolicity constant (Proposition \ref{pro:unif hyp}) and the fact that horoballs are quasi-convex with constant uniformly controlled by the hyperbolicity constant (see e.g. \cite[Lemma 3.10]{GM}). The intersection property follows from an application of \cite[Lemma 4.5]{DS}. To apply it, we need to check that the properties $(\alpha_1)$ and $(\alpha_2)$ of \cite[Lemma 4.5]{DS} are satisfied in our setting. Both of them hold by the following claim, which is well-known within a single cusped space $X_k$ (see e.g. \cite[Lemma 3.1]{MS:qhyp}). Here, however, we need constants independent of $k$, so we provide an argument.

\begin{claim}
There exists $D>0$ such that for all $R>0$ the following holds. For any two distinct horoballs $H_1, H_2$ in some $X_k$, we have that $N_R(H_1)\cap H_2$ has diameter at most $DR+D$.
\end{claim} 

\begin{proof}
It suffices to prove the same statement replacing $H_1$, $H_2$ with the convex sub-horoballs described above. Consider points $p_i,q_i\in H_i$ with $d_{X_k}(p_1,p_2),d_{X_k}(q_1,q_2)\leq R$. It follows easily from the description of geodesics in horoballs, see \cite[Lemma 3.10]{GM}, that if $d_{X_k}(p_2,q_2)>10R+10\delta+10$, then a geodesic from $p_2$ to $q_2$ contains a point at height larger than $R+2\delta$. Such a point cannot be $2\delta$ close to another side of a quadrangle with vertices $p_1,q_1,q_2,p_2$, contradicting $\delta$-hyperbolicity. Therefore, any $p_2,q_2$ as above will have $d_{X_k}(p_2,q_2)\leq 10R+10\delta+10$, as required.
\end{proof}

This finishes the proof of the lemma.
\end{proof}

Peripheral arcs are the only source of potential pathological behaviors of the arc stabilisers. In fact, we have the following. 

\begin{lemma}
\label{lem:stabs_in T}
Stabilisers of non-trivial non-peripheral arcs in $T$ are either trivial or maximal cyclic subgroups of $G$.
\end{lemma}

\begin{proof}
Consider a non-trivial non-peripheral arc $[a,b]$ in $T$, and assume that its stabiliser ${\rm Stab}([a,b])$ is non-trivial. Consider a non-trivial element $h\in {\rm Stab}([a,b])$. Note that $h$ has infinite order (since $G$ is torsion-free). There are two cases.

Denote by $P_{k(n)}$ the image of $\pi_1(U-B)$ in $Q_{k(n)}$ (see Section \ref{sec:rel hyp}).

{\slshape Case (1)}. For infinitely many $n$, the element $\phi_n(h)$ is not conjugate into $P_{k(n)}$. 

In this case, we follow word-for-word the proof of \cite[Lemma 5.11]{DG18}, with a single modification. In the proof, Dahmani and Guirardel use the fact that no loxodromic element can stabilise a horoball in the cusped spaces they consider. We replace this fact with the following: No element which is not conjugate into $P_{k(n)}$ (such as $\phi_n(h)$, by our assumption) can stabilise a horoball of $\mc{H}_{k(n)}$ in $X_{k(n)}$ (see Section \ref{sec:rel hyp}).

{\slshape Case (2)}. For infinitely many $n$, the element $\phi_n(h)$ is conjugate into $P_{k(n)}$. 

We actually exclude this case. The key property to rule it out is the following. 

\begin{claim}
There exists $C>0$ such that for all sufficiently large $k$ the following holds. Let $H_k\subset X_k$ be the horoball stabilised by $P_k$. For all $g\in P_k$ non-trivial and $x\in X_k$, we have $d_{X_k}(x,g^jx)\geq d_{X_k}(x,H_k)$ for some $j$ with $|j|\leq C$.
\end{claim}

\begin{proof}
If $k$ is sufficiently large, we can find a uniformly bounded power $g^j$ of $g$ that translates every point of the boundary horosphere of $H_k$ a distance much larger than the hyperbolicity constant of $X_k$ (which, recall, is bounded independently of $k$ by Proposition \ref{pro:unif hyp}). Considering an ideal triangle with endpoints $x,g^jx$, and the point at infinity for $H_k$, we then see that some geodesic from $x$ to $g^jx$ passes through $H_k$. The conclusion follows.
\end{proof}

We can now rule out Case (2). Fix $\ep>0$. Consider a point $x\in[a,b]$. For $n$ sufficiently large, choose approximating elements $x_n\in X_{k(n)}$ such that for every $|j|\le 2C$ (where $C$ is as in the above claim) we have
\[
\left|\frac{1}{d_n}d_{X_{k(n)}}(x_n,\phi_n(h)^jx_n)-d_T(x,h^jx)\right|<\ep
\]
as provided by the equivariant Hausdorff-Gromov convergence (see Definition \ref{dfn:hg conv}). Note that $d_T(x,h^{2C}x)=0$ as $h^{2C}$ fixes pointwise $[a,b]$. Thus
\[
d_{X_{k(n)}}(x_n,\phi_n(h)^{2C}x_n)/d_n<\ep
\]
By assumption, $\phi_n(h)$ stabilises a horoball $H_n\in\mc{H}_{k(n)}$. Hence, by the claim, we have
\[
d_{X_{k(n)}}(x_n,H_n)/d_n\le d_{X_{k(n)}}(x_n,\phi_n(h)^{2C}x_n)/d_n<\ep.
\]
This implies that the sequence of horoballs $H_n$ stabilised by the $\phi_n(h)$'s defines a limit horoball $H$ (see Definition \ref{dfn:lim hor}) which is invariant under $h$. Moreover, it also implies that $d_T(x,H)\le\ep$. As $\ep$ is arbitrary and $H$ is closed, we get that $x\in H$. As $x\in[a,b]$ was arbitrary, we deduce that $[a,b]\subset H$, which contradicts the fact that $[a,b]$ is non-peripheral. This concludes the proof of the lemma.
\end{proof}

We denote by $T_0$ the minimal subtree for $T$, meaning the smallest closed non-empty subtree invariant under the $G$-action. There are 3 possible scenarios. Recall that $\mathcal P$ denotes the collection of all limit horoballs.

\begin{enumerate}[label=(\arabic*)]
    \item\label{item:case1} {\bf Single intersection}. There exists a non-trivial arc $\alpha$ in $T_0$ such that for every $H\in\mathcal P$
    \[
    \alpha\cap H
    \]
    is either empty or a single point.
    \item\label{item:case2} {\bf Finitely covered arcs}. For every non-trivial arc $\alpha$ of $T_0$ there are $H_1,\cdots,H_r\in\mc{P}$ such that
    \[
    \alpha\subset H_1\cup\cdots\cup H_r.
    \]
    \item\label{item:case3} {\bf Generically dense with a non-finitely covered arc}. For any arc $\alpha$ in $T_0$, we have that 
    \[
    \bigcup_{H\in \mathcal P: |H\cap \alpha|>1} H\cap \alpha
    \]
    is dense in $\alpha$, but there exists a non-trivial arc of $T_0$ not covered by finitely many elements of $\mathcal P$.
\end{enumerate}

Note that at least of the above must hold. Indeed, if (1) and (2) do not hold, then any arc $\alpha$ in $T_0$ cannot have a non-trivial subarc that only intersects elements of $\mathcal P$ in at most one point. That is, (3) holds.

\subsection{Single intersection}
We discuss Case \ref{item:case1}.

\begin{pro}
\label{prop:pseudo-metric}
In Case \ref{item:case1}, $G$ acts without a global fixed point on a tree $T_1$ such that each arc stabiliser is either trivial or a maximal cyclic subgroup. 
\end{pro}

\begin{proof}
We will use that $T_0$ is {\em tree-graded} in the sense of \cite[Definition 2.1]{DS}, with {\em pieces} given by the $H\cap T_0$ for $H\in\mathcal P$. Indeed, the conditions for a complete geodesic metric space (note that $T_0$ is complete because it is closed in $T$ and $T$ is complete) being tree-graded are that (T1) any two pieces intersect in at most one point, which is the case for elements of $\mathcal P$ by Lemma \ref{lem:hor_inters}, and that (T2) all simple loops are contained in a piece, which holds vacuously for us.

Denote by $d$ the metric of $T_0$. For $x,y\in T_0$, define $\tilde d(x,y)$ to be $d(x,y)$ minus the length of the intersections of the geodesic $[x,y]$ with elements of $\mathcal P$. In \cite[Section 2]{DS2} it is shown that this is a pseudo-metric, and that considering the equivalence relation where $x\sim y$ if $\tilde d(x,y)=0$, we have that 
\[
T_1:=\left(T_0/_\sim\,,\,\tilde d\right)
\]
is an $\mathbb R$-tree, see in particular \cite[Lemma 2.19]{DS2}.
    
The hypothesis of Case \ref{item:case1} is easily seen to imply that $T_1$ is not a single point. Moreover, we claim that the stabiliser of any non-trivial arc $\alpha$ in $T_1$ coincides with the stabiliser of a non-peripheral arc of $T_0$. Indeed, since $T_1$ is a tree, the stabiliser of $\alpha$ coincides with the stabiliser of its endpoints, which in turn coincides with the common stabiliser of the pre-images of the endpoints in $T_0$. Said pre-images are (closed and connected, so they are) closed sub-trees by \cite[Lemmas 2.15 and 2.17]{DS2}, and they are necessarily disjoint. Therefore, the common stabiliser of these sub-trees coincides with the stabiliser of a minimal arc joining them. This arc cannot be peripheral, for otherwise the endpoints of $\alpha$ would coincide, and we are done by Lemma \ref{lem:stabs_in T}.
\end{proof}

\subsection{Finitely covered arcs}
We discuss Case \ref{item:case2}.

\begin{pro}
\label{prop:simplicial}
In Case \ref{item:case2}, either $G$ locally covers $B$ or $G$ acts on a simplicial tree $T_2$ without a global fixed point and with edge stabilisers that locally cover $B$.
 \end{pro}

\begin{proof}
Consider the set $\mathcal P_0\subset\mc{P}$ of limit horoballs (see Definition \ref{dfn:lim hor}) that intersect $T_0$ in more than one point. Note that $\mathcal P_0$ cannot be empty, for otherwise $T_0$ would need to be a single point, but $G$ acts with displacement $1$ on $T_0$. Let $\mathcal I$ be the collection of all points of $T_0$ that are intersections of two elements of $\mathcal P_0$. Consider the graph $T_2$ with vertex set $\mathcal P_0\cup \mathcal I$ and where an element of $\mathcal I$ is connected to the horoballs of $\mathcal P_0$ in which it is contained.

\begin{claim}
$T_2$ is a simplicial tree.
\end{claim}

\begin{proof}
First of all, we argue that $T_2$ is connected. To see this, it is enough to argue that any two vertices $v,v'\in\mc{P}_0$ can be connected by a path. By definition, $v,v'$ correspond to limit horoballs $H,H'\subset T$ that intersect $T_0$ in more than one point. Let $\alpha\subset T_0$ be an arc connecting a point in $H\cap T_0$ to a point in $H'\cap T_0$. By the assumption of Case \ref{item:case2}, we have that $\alpha$ is contained in a finite union of limit horoballs 
\[
\alpha\subset H_1\cup\cdots\cup H_r\Longrightarrow\alpha=(H_1\cap \alpha)\cup\cdots\cup(H_r\cap\alpha).
\]
Each $H_j\cap\alpha$ is a connected closed subsegment of $\alpha$. By connectedness, $\alpha$ can be covered using just those $H_j\cap \alpha$ that are not empty and not a single point (in particular, we can assume that every $H_j$ is in $\mc{P}_0$). By connectedness of $\alpha$, it also easily follows that the subgraph of $T_\alpha\subset T_2$ spanned by $H_1,\cdots,H_r\in\mc{P}_0$ and $\bigcup (H_i\cap H_j\cap\alpha)\subset\mc{I}$ is connected. As $H,H'$ are connected to some vertices of $T_\alpha$ (they must intersect some $H_j$'s), they can be joined by a path in $T_2$, as required.

Let us now discuss cycles in $T_2$. If $T_2$ contained a simple loop $\ell$, we could construct a loop in $T_0$ by concatenating geodesics in elements of $\mathcal P_0$ that occur along $\ell$. Since distinct elements of $\mathcal P$ intersect in at most one point (Lemma \ref{lem:hor_inters}), this loop in $T_0$ is easily seen to be simple as well, contradicting that $T_0$ is a tree.
\end{proof}

\begin{claim}
If $v\in\mathcal P_0$, then the stabiliser of $v$ locally covers $B$.
\end{claim}

\begin{proof}
Recall that $v$ corresponds to a limit of horoballs $H_n\in\mc{H}_{k(n)}$ with 
\[
\sup{d_{X_{k(n)}}(o_n,H_n)/d_n}=D<\infty.
\]
Consider $h_1,\cdots,h_r\in{\rm Stab}(v)$ and $x\in H$ with $d_T(o,h_jx)>4D$ for every $j\le r$. We have $h_jx\in H$ for every $j\le r$. Fix $\ep<D$. By Hausdorff-Gromov convergence (see Definition \ref{dfn:hg conv}), for every $n$ large enough we find $x^n,y_j^n\in H_n$ such that 
\[
d_{X_{k(n)}}(y_j^n,\phi_n(h_j)x^n)/d_n<\ep\quad\text{\rm and}\quad d_{X_{k(n)}}(o_n,y_j^n)/d_n\ge d_T(o,h_jx)-\ep\ge 2D.
\]
for every $j\le r$. We argue that, for $n$ large enough, we have $\phi_n(h_j)x_n\in H_n$ for every $j\le r$ and, hence, $\phi_n(h)H_n\cap H_n\neq\emptyset$. As the collection $\mc{H}_{k(n)}$ is $Q_{k(n)}$-precisely invariant we deduce that $\phi_n(h_1),\cdots,\phi_n(h_r)\in{\rm Stab}(H_n)$ for all $n$ sufficiently large. To see that $h_jx_n\in H_n$, observe that $\phi_n(h_j)x_n$ is in the $\ep$-neighborhood (for the metric $d_{X_{k(n)}}/d_n$) of the point $y_j^n$ which is contained in the horoball $H_n$. Note also that $y_j^n$ has distance from the origin $o_n$ at least $2D$ where $D\ge d_{X_{k(n)}}(o_n,H_n)/d_n$. As $2D-\ep>D$, the $\ep$-neighborhood of $y_j^n$ for the metric $d_{X_{k(n)}}/d_n$ is entirely contained in the horoball $H_n$.
Consider now a finitely generated subgroup $H<{\rm Stab}(v)$. By the above discussion, for some $n$ we have $\phi_n(H)$ lies in the stabiliser ${\rm Stab}(H_n)$ of the horoball $H_n$, which is a conjugate of $P_{k(n)}$ (see Section \ref{sec:rel hyp}), because the generators of $H$ are all mapped to ${\rm Stab}(H_n)$. Since $P_{k(n)}$ is isomorphic to $ \pi_1(B)\times\mathbb Z/k(n)\mb{Z}$ (see Section \ref{sec:rel hyp}), $\phi_n$ is injective, and $H$ is torsion-free, we have that $H$ is isomorphic to a subgroup of $\pi_1(B)$, as required.
\end{proof}

If $G$ does not fix any vertex of $T_2$, we are done. Suppose then that $G$ fixes a vertex $v$ of $T_2$. If $v\in\mathcal P_0$, then $G$ locally covers $B$, and we are done. On the other hand, if $v\in \mathcal I$, then $G$ does not fix $v$ because the action on $T_0$ does not have a fixed point. This concludes the proof.
\end{proof}

\subsection{Generically dense with a non-finitely covered arc}
Finally, we discuss Case \ref{item:case3}.

\begin{pro}
\label{prop:pass-to-cone-off}
In Case \ref{item:case3}, $G$ acts without a global fixed point on a tree $T_3$ such that the stabilisers of unstable arcs have bounded cardinality. 
\end{pro}

\begin{proof}
The idea is to apply \cite[Theorem 4.4, Lemma 4.7-(4)]{GrovesHull} which gives the desired conclusion provided that we prove two properties. First, we need to find a family of {\em uniformly acylindrical} actions $Q_k\curvearrowright\bar{X}_k$ (in the sense of \cite[Definition 2.8]{GrovesHull}). Then, we need to show that, with respect to these actions, the injective homomorphisms $\phi_n:G\to Q_{k(n)}$ are {\em divergent} (in the sense of \cite[Definition 4.2]{GrovesHull}).

For us, $\bar X_{k}$ will be the {\em coned-off graph} for $Q_k$, namely the space obtained from $X_k$ by adding edges connecting all pairs of vertices contained in the same horoball $H_k\in\mc{H}_k$. 

We now establish the uniform acylindricity property of $Q_k\curvearrowright\bar X_k$. 

\begin{lemma}
\label{lem:unif_acyl}
There exists $\delta'\geq 1$ such that all $\bar X_k$ are $\delta'$-hyperbolic. Moreover, the actions of $Q_k$ on $\bar X_k$ are uniformly acylindrical.
\end{lemma}

\begin{proof}
Uniform hyperbolicity follows from the well-known fact that coning off uniformly quasi-convex subsets of a hyperbolic space yields a hyperbolic space, and the new hyperbolicity constant depends on the old one and the quasiconvexity constant of the subsets. This follows, for instance, from \cite[Corollary 2.4]{Kapovich-Rafi}.

To show acylindricity, we actually need a bit more.
    
\begin{claim}
\label{claim:replace}
There exists a constant $\lambda$ such that for all $k$ and all geodesic $\alpha$ in $X_k$, we can replace a collection of subsegments of $\alpha$ contained in horoballs with single edges to obtain a $(\lambda,\lambda)$-quasi-geodesic $\bar\alpha$ of $\bar X_k$ (compare with \cite[Proposition 7.9]{Hruska}).
\end{claim}

\begin{proof}[Proof of the Claim]
To see this, note that, given such a geodesic $\alpha$, we can find a disjoint collection of subsegments that uniformly coarsely coincide with the intersection of $\alpha$ with all horoballs it passes through. Denoting $\bar\alpha$ the path obtained by replacing those with a single edge, we can construct a coarse retraction onto $\bar\alpha$ using closest-point projection to $\alpha$ in $X_k$; we use here that if the projection of a horoball onto $\alpha$ is sufficiently large, then $\alpha$ passes through the horoball. The existence of this retraction guarantees that $\bar\alpha$ is a uniform quasi-geodesic.
\end{proof}

For a geodesic $\alpha$, denote by $\alpha_{\rm tr}$ the complement of the subsegments as in Claim \ref{claim:replace}.
    
\begin{claim}
\label{claim:transient_thin}
There exists a constant $D$ such that, for all $k$, given a triangle $\alpha\cup\beta\cup \gamma$ in $X_k$ we have $\alpha_{\rm tr}\subseteq N_D(\beta_{\rm tr}\cup\gamma_{\rm tr})$ (compare with \cite[Definition 3.11]{metric_rh}). 
\end{claim}

\begin{proof}[Proof of the Claim]
Given $p\in \alpha_{\rm tr}$, we know by hyperbolicity that it is $\delta$-close to either $\beta$ or $\gamma$. It is easy to see that a point of, say, $\beta$ which is within bounded distance of the complement of the horoballs is uniformly close to $\beta_{\rm tr}$, so we are done.
\end{proof}

Now, the constants implicit in the following arguments can be chosen independently of $k$, and this will conclude the proof. Consider two sufficiently far away points $x,y\in\bar X_k$, and an element $g$ moving them both a bounded amount. There is a quasi-geodesic $\bar\alpha$ connecting $x,y$ as in Claim \ref{claim:replace}. Consider a point $p$ of $\alpha_{\rm tr}$ such that $p$ is lies at the same distance in $\bar X_k$ from $x$ and $y$, up to a bounded error. Using two geodesic triangles and Claim \ref{claim:transient_thin}, it is readily seen that $p$ is $2D$-close to a point $gq$ of $g\alpha_{\rm tr}$. Moreover, $q$ lies within bounded distance of $p$ in $\bar X_k$ (but possibly not in $X_k$). There are boundedly many possibilities for $q$ since $\bar\alpha$ is a quasi-geodesic in a hyperbolic space, and boundedly many possibilities for $gq$ since balls in $X_k$ are finite. Since the stabiliser of $q$ is finite, there are boundedly many possibilities for $g$, as required. 

This finishes the proof of the lemma.
\end{proof}

Next, we prove that the sequence of actions $\phi_n:G\to Q_{k(n)}$ is divergent.

\begin{lemma}
\label{lem:divergent}
The sequence of actions $\phi_n:G\to Q_{k(n)}$ is divergent for $Q_{k(n)}\curvearrowright X_{k(n)}$.    
\end{lemma}

\begin{proof}
Consider an arc $\alpha\subset T_0$ as provided by Case \ref{item:case3}, that is, an arc that cannot be covered by finitely many limit horoballs $H\in\mc{P}$. It will be convenient to arrange the arc $\alpha$ to be contained in the axis of some element of $g$. This is what we do with the next claim.

\begin{claim}
\label{claim:g}
There exists an arc $\beta$ as in (3) which is contained in the axis of a loxodromic element $g$ for the action of $G$ on $T_0$.
\end{claim}

\begin{proof}[Proof of the Claim]
Fix an arc $\alpha$ as in (3). Since the action of $G$ on $T_0$ is minimal, it is well-known that $T_0$ is the union of all axes of hyperbolic elements. In particular, there exist loxodromic elements $g_0,g_1$ whose axes contain the endpoints of $\alpha$. If either axis contains $\alpha$, we are done. If not, but the axes intersect, then one of them must contain a sub-arc of $\alpha$ which still cannot be covered by finitely many elements of $\mathcal P$, so we are done in this case as well. More generally, by the same argument, we can also assume that neither of the axes contains a subarc of $\alpha$ that cannot be covered by finitely many elements of $\mathcal P$. In particular, the minimal arc $\beta$ joining the axes of $g_0$ and $g_1$ cannot be covered by finitely many elements of $\mathcal P$, and in particular is as in (3). Since the axes of $g_0$ and $g_1$ are disjoint, it is well-known that the product $g_0g_1$ will be loxodromic and that its axis contains $\beta$, and this concludes the proof.
\end{proof}

Using the arc $\beta$ of Claim \ref{claim:g}, we prove the desired divergence property.

\begin{claim}
\label{claim:diverge_cone-off}
The sequence of translation lengths of $\phi_n(g)$ on $\bar X_{k(n)}$ diverges.
\end{claim}

\begin{proof}[Proof of the Claim]
Consider $g$ and $\beta$ as in Claim \ref{claim:g}. Up to passing to a power of $g$, we can assume that $g\beta$ and $\beta$ are disjoint. Since there are infinitely many non-trivial subarcs of $\beta$ of the form $\beta\cap H$, we can deduce the following.  For all integers $k\geq 1$ and constant $C\geq 1$ the following holds for sufficiently large $n$:
\begin{itemize}
    \item $\phi_n(g)$ is loxodromic for the action on $X_{k(n)}$,
    \item there exist horoballs $H^n_1,\dots, H^n_k$ of $X_{k(n)}$ such that a uniform quasi-axis $\gamma_n$ of $\phi_n(g)$ intersects each $H^n_j$ in a subgeodesic of length at least $C$, and 
    \item $\gamma_n$ has two disjoint subgeodesics, one containing the intersection of $\gamma$ with all $H^n_j$ and the other one containing the intersection of $\gamma$ with all $\phi_n(g)H^n_j$.
\end{itemize}
    
The properties above, for $C$ sufficiently large, guarantee that the translation length of $\phi_n(g)$ on $\bar X_{k(n)}$ is bounded from below linearly in $k(n)$ by Claim \ref{claim:replace}. Therefore, the sequence of translation lengths diverges as required.
\end{proof}

This finishes the proof of the lemma.
\end{proof}

Because of Lemma \ref{lem:unif_acyl} and Lemma \ref{lem:divergent}, we can apply \cite[Theorem 4.4, Lemma 4.7-(4)]{GrovesHull}. This finishes the proof of the proposition.
\end{proof}

\subsection{Assembling the proof of Theorem \ref{thm:displ}}
We can now conclude the proof of Theorem \ref{thm:displ}. Recall that we have to argue that $G$ admits a suitable splitting. In Case \ref{item:case2}, we directly apply Proposition \ref{prop:simplicial}, as in this case the action is on a simplicial tree. In the other two cases, we have actions on real trees, and we apply suitable versions of the Rips machine. Specifically, in Case \ref{item:case1} we get the required splitting by applying \cite[Theorem 9.6]{BF95} to the action on a real tree coming from Proposition \ref{prop:pseudo-metric}, while in Case \ref{item:case3} we combine Proposition \ref{prop:pass-to-cone-off} and \cite[Main Theorem]{Guirardel}.

\section{Sectors do not split}
\label{sec:no_split}

Recall that, given a GT-pair $(M,B)$ with wall $W$, we defined its associated complete sector $\bar{S}$ as the metric completion of the complement of a wall $S=M-W$ (see Definition \ref{dfn:sector}). Note that $\pi_1(S)=\pi_1(\bar{S})$. The goal of this section is to show that $\pi_1(\bar{S})$ does not split over any subgroup that locally covers $B$ (see Definition \ref{dfn:locally covers}). 

\begin{pro}
\label{prop:no_split}    
Let $(M,B)$ be a GT-pair of dimension $d\geq 3$ with wall $W$. Then $\pi_1(\bar{S})$ does not split over subgroups that locally cover $B$.
\end{pro}

\begin{proof}
Suppose that $\pi_1(\bar{S})=G_1*_HG_2$ or $\pi_1(\bar{S})=G_1*_H$ with $H$ that locally covers $B$, and $H$ a proper subgroup of $G_j$. Denote by $\pi_1(\bar{S})\curvearrowright T$ the associated action on the Bass-Serre tree $T$. 

Consider the action of the subgroup $\pi_1(\partial\bar{S})<\pi_1(\bar{S})$. There are two cases. Either $\pi_1(\partial\bar{S})$ fixes a vertex (and, hence, it is conjugate in one of the factors $G_j$, say $G_1$) or we get a decomposition of $\pi_1(\partial\bar{S})$ as a graph of groups. We rule out the first case by homological arguments and the second case by geometric ones.

{\slshape Case (1)}. Suppose $\pi_1(\partial\bar{S})<G_1$. We double $\bar{S}$ along the boundary and obtain the closed $d$-manifold $D\bar{S}$. As $\pi_1(\partial\bar{S})<G_1$, in the amalgamated product case we get a splitting of $\pi_1(D\bar{S})=G_1'*_HG_2$ where $G_1'=G_1*_{\pi_1(\partial\bar{S})}\pi_1(\bar{S})$. The HNN case is very similar, and we do not spell out the details in that case. By Mayer-Vietoris (homology coefficients in $\mb{Z}$), we get the short exact sequence
\[
\xymatrix{
H_{d}(H)\ar[r] &H_d(G_1')\oplus H_d(G_2)\ar[r] &H_d(\pi_1(D\bar{S}))\ar[r]  &H_{d-1}(H).
}
\]
To conclude, we now apply the following lemma that we prove below.

\begin{lemma}
\label{lem:finite_infinite_index}
Let $H$ be a group that locally covers $B$. Then either:
\begin{enumerate}
    \item{$H$ is isomorphic to a finite index subgroup of $\pi_1(B)$.}
    \item{The homological dimension of $H$ is at most $d-3$.}
\end{enumerate}
\end{lemma}

Assuming Lemma \ref{lem:finite_infinite_index}, we finish the proof of Case (1) as follows. By Lemma \ref{lem:finite_infinite_index}, we have $H_{d-1}(H)=H_d(H)=0$. Since $D\bar{S}$ is closed and orientable $H_d(\pi_1(D\bar{S}))=\mb{Z}$. So necessarily either $H_d(G_1')=\mb{Z}$ or $H_d(G_2)=\mb{Z}$. However, both $G_1',G_2$ are infinite index subgroups (this always happens in an amalgamated product if the amalgamating subgroup is proper in both factors) of $\pi_1(D\bar{S})$, so their $d$-dimensional homologies vanish, a contradiction.

It remains to prove Lemma \ref{lem:finite_infinite_index}.

\begin{proof}[Proof of Lemma \ref{lem:finite_infinite_index}]
Since $B$ is a closed aspherical manifold, it is well-known that the homological dimension of a subgroup of $\pi_1(B)$ is $d-2$ when the subgroup has finite index and $\le d-3$ when it has infinite index.

If $H$ is finitely generated, then $H$ is isomorphic to a subgroup of $\pi_1(B)$ by definition of locally covering. Thus the conclusion follows from the previous observation.

If $H$ is not finitely generated, then it is a strictly increasing countable union of finitely generated subgroups $H=\bigcup_{j\in\mb{N}}{H_j}$. By Definition \ref{dfn:locally covers}, each $H_j$ is isomorphic to a subgroup $F_j<\pi_1(B)$ which has finite index if $H_j$ has homological dimension $d-2$ and infinite index if $H_j$ has homological dimension $\le d-3$. 

There are two scenarios. Either there exists $j$ such that $H_j$ has homological dimension $d-2$ or every $H_j$ has homological dimension $d-3$. To analyse them, we make the following observation. 

If $F_j$ has finite index (equivalently, if $H_j$ has homological dimension $d-2$), the index of $F_j$ in $\pi_1(B)$ is completely determined by $H_j$ (it is independent of the embedding $H_j\to\pi_1(B)$). Indeed, the degree is determined by the ratio between the simplicial volumes $||H_j||/||\pi_1(B)||$ (for the definition and properties of the simplicial volume, we refer to \cite{Gromov} or \cite[Chapter 7]{Frigerio}). 

Suppose we are in the first scenario, that is, we have $j_0$ such that $H_{j_0}$ has homological dimension $d-2$. Any embedding $f_j:H_j\to F_j<\pi_1(B)$ with $j\ge j_0$ restricts to an embedding $f_j:H_{j_0}\to\pi_1(B)$. By the observation above, the index of $f_j(H_{j_0})<F_j$ in $\pi_1(B)$ is completely determined by $H_{j_0}$. We deduce that $F_j$ is a finite-index subgroup of $\pi_1(B)$ with index bounded by the index of $F_{j_0}$. As there are only finitely many finite-index subgroups up to a fixed index, $F_j$ is constant for infinitely many $j\ge j_0$ and, hence, we must have $H_j$ constant for all large $j$, but then $H$ is finitely generated and coincides with $H_j$ for all $j$ large enough. This contradicts our assumption that $H$ is not finitely generated. So the first scenario does not happen under this hypothesis.

Suppose we are in the second scenario, that is, each $H_j$ has homological dimension $\le d-3$. As the homological dimension of $H$ is at most the supremum of the homological dimensions of the $H_j$'s (see e.g. \cite[Theorem 4.7]{Bieri}) and each $H_j$ has homological dimension $\le d-3$ because it is isomorphic to an infinite index subgroup of $\pi_1(B)$, we conclude that $H$ has homological dimension $\le d-3$ as desired.
\end{proof}

Next, we discuss Case (2).

{\slshape Case (2)}. Consider the splitting of $\pi_1(\partial \bar{S})$ as a graph of groups. The edge stabilisers are subgroups of $H$. Note that the homological dimension of $\pi_1(\partial \bar{S})$ is $d-1$ and, by Lemma \ref{lem:finite_infinite_index}, each edge stabiliser has either homological dimension $\le d-3$ or it is isomorphic to a finite-index subgroup of $\pi_1(B)$. By the same argument of Case (1), shifting degrees down by 1, we cannot split $\pi_1(\partial \bar{S})$ along subgroups of homological dimension $\le d-3$. So the only possibility left is that each edge stabiliser is isomorphic to a finite index subgroup of $\pi_1(B)$. This means in particular that any edge stabiliser for $\pi_1(\partial \bar{S})$ is a finite index subgroup of an edge stabiliser for $\pi_1(\bar{S})$. 

Therefore, to conclude Case (2), we are left to prove the following.

\begin{lemma}
\label{lem:coarse_split}
$\pi_1(\bar{S})$ does not split over any subgroup virtually contained in $\pi_1(\partial \bar{S})$.
\end{lemma}

\begin{proof}[Proof of Lemma \ref{lem:coarse_split}]
For the proof, we recall a definition and a fact.

\begin{dfn}[Coarse Separation]
Given a subset $A\subseteq Y$ of a metric space $Y$, we say that $A$ {\em coarsely separates} $Y$ if for all $R>0$ there exists $R'>0$ such that given $x,y\in Y$ with $d_Y(x,A),d_Y(y,A)>R'$ there exists a path joining $x$ to $y$ outside the $R$-neighborhood of $A$.    
\end{dfn}

By e.g. \cite[Lemma 2.2]{Pap}, if a finitely generated group $G$ splits over a subgroup $H$, then $H$ coarsely separates $G$, where we conflate $G$ with one of its Cayley graphs.

In view of this, we will argue that no subset of $\pi_1(\partial \bar{S})$ can coarsely separate $\pi_1(\bar S)$. 

As coarse separation has suitable quasi-isometry invariance properties, it suffices to argue that $Y$, the universal cover of $\bar S$ endowed $\bar S$ with the metric from Corollary \ref{cor:gromov hyperbolic}, is not coarsely separated a subset of a boundary component $C\subset\partial Y$ (note that the boundary component is convex, by the same corollary).

In order to prove this, we begin with the following observation.

\begin{claim}
Any point $x\in Y$ lies on a uniform quasi-geodesic ray $\gamma_x$ starting at its closest-point projection $\pi(x)$ in $C$, and such that the distance to $C$ diverges along $\gamma_x$.
\end{claim}
 
\begin{proof}[Proof of the Claim]
Indeed, we can prolong the geodesic $[\pi(x),x]$ until either we obtain a geodesic ray, or we hit another boundary component, and append a geodesic ray contained in said component forming an angle bounded away from $0$ with $[\pi(x),x]$. 
\end{proof}

We note that there exists a constant $K$ such that the distance of any $x'\in \gamma_x$ to $A$ is $d_Y(x',\pi(x))+d_Y(\pi(x),A)$ up to an additive error of at most $K$. This is because any geodesic from $x'$ to $A$ passes uniformly close to $\pi(x)$ due to well-known properties of closest-point projections to (quasi)convex subspaces of hyperbolic spaces.

Now, fix $R>0$, and consider $x,y\in Y$ at distance at least $R$ from $A$. By the above, the subrays of $\gamma_x$ and $\gamma_y$ starting at $x$ and $y$ lie outside the $(R-2K)$-neighborhood of $C$. Thus, to prove that $C$ does not coarsely separate $Y$, it suffices to show that points $x',y'$ sufficiently far along $\gamma_x,\gamma_y$ can be connected by a path outside the $R$-neighborhood of $C$. To do so, we exploit the following.

\begin{claim}
\label{claim:sierpinski}
Let $Z=\partial_\infty Y-\partial_\infty C$ be the Gromov boundary $\partial_{\infty}Y$ of $Y$ minus the limit set $\partial_\infty C$ of $C$. Then $Z$ is connected.
\end{claim}

\begin{proof}[Proof of the Claim]
Note that the set of points of $\partial_{\infty}Y$ not contained in the union $U$ of all limit sets of $\pi_1(\bar S)$-translates of $C$ is dense in $\partial_{\infty}Y$.

By \cite{relations_boundaries}, $\partial_{\infty}Y-U$ is homeomorphic to the complement $\mb{B}-P$ in the Bowditch boundary $\mb{B}$ of the relatively hyperbolic pair $(\pi_1(\bar S),\pi_1(\partial \bar S))$ of the union $P$ of all parabolic points. By \cite{cohom_bowditch}, $\mb{B}$ is a \v{C}ech homology $(n-1)$-sphere, and therefore $\mb{B}-P$ is connected ($\mb{B}$ remains connected after removing finitely many points and an intersection of connected sets is connected). Hence so is $\partial_{\infty}Y-U$, and in turn so is $Z$ since $\partial_{\infty}Y-U$ is dense.
\end{proof}

Let $x_\infty$ and $y_\infty$ be the limit points of $\gamma_x$ and $\gamma_y$. By Claim \ref{claim:sierpinski}, for any $\epsilon>0$ there exists a sequence $x_\infty=a_0,\dots, a_n=y_\infty$ with consecutive elements lying within distance at most $\epsilon$ (in a fixed visual metric on the boundary). For $\epsilon>0$ small enough and $L>0$ large enough, we can choose points $b_j\in[\pi(a_j),a_j)$ such that
\begin{itemize}
    \item $a_0,a_n$ lie within $L$ of points $x',y'$ along $\gamma_x,\gamma_y$ at distance $\geq R+L$ from $C$,
    \item $d_Y(b_j,b_{j+1})\leq L$ for $j\le n$,
    \item $d_Y(b_j,C)\geq R+L$.
\end{itemize}

Connecting the $x'$, $y'$, and $b_j$ in the appropriate order with geodesics, we obtain the desired path outside the $R$-neighborhood of $C$. 
\end{proof}

This concludes the proof of Case (2) and, hence, of Proposition \ref{prop:no_split}.
\end{proof}

\section{Outer automorphism groups}
\label{sec:proof main}

In this section, we prove Theorem \ref{thmA:main}, Theorem \ref{thmA:Out}, and Theorem \ref{thmA:Sylow}. These all require understanding (outer) automorphisms of fundamental groups of Gromov-Thurston manifolds.

We begin by recalling some notation. Fix a GT-pair $(M,B)$, of dimension $d\geq 3$, with wall $W$. We denote by $\bar{S}$ its associated complete sector, that is, the completion of $M-W$ with respect to the intrinsic path metric. Its boundary $\partial\bar{S}$ consists of two isometric copies of $W$, denoted by $W^+,W^-$ intersecting along the singularity $B=W^-\cap W^+$. The degree $k$ branched cover of $M$ along $B$ is obtained gluing $k$ copies $\bar{S}_j$ of $\bar{S}$ along their walls $W_j^\pm$, that is
\[
M_k:=\bigsqcup_{j\in\mb{Z}/k\mb{Z}}{\bar{S}_j}\left/\,\{W_j^+=W_{j+1}^-:j\in\mb{Z}/k\mb{Z}\}\right..
\]
The intersection of all sectors $\bar{S}_j$ is $B_k$, a copy of $B$. We fix once and for all a basepoint $x_k\in B_k$, mapping to a fixed point $x\in B$ under the branched covering $p_k:M_k\to M$. 

A key role in the proofs of our main theorems is played by the group
\[
G_k={\rm Out}(\pi_1(M_k,x_k))
\]
which has finite order (by \cite{Paulin} plus the well-known fact that fundamental groups of closed aspherical $d$-manifolds with $d\ge 3$ do not split over cyclic subgroups). Inside $G_k$ we have the order $k$ element (by Corollary \ref{cor:order_k}) induced by the canonical rotation $\rho_k:M_k\to M_k$ sending $\rho_k(\bar{S}_j)=\bar{S}_{j+1}$ (as in Definition \ref{dfn:branched cover}). We conflate $\rho_k$ with the automorphism of $\pi_1(M_k,x_k)$ that it induces.

\subsection{The action of $G_k$ on injections and Theorem \ref{thmA:main}}

Denote by $\mc{C}_k$ the set of conjugacy classes of injections $\pi_1(\bar{S})\to \pi_1(M_k,x_k)$. These include the canonical inclusions (see Lemma \ref{lem:pi1injective}) induced by $\bar{S}_j\subset M_k$ which we denote by
\[
\iota_j:\pi_1(\bar{S})\to\pi_1(\bar{S}_j,x_k)<\pi_1(M_k,x_k).
\]
Note that $G_k$ acts naturally on $\mc{C}_k$. We have the following.

\begin{lemma}
\label{lem:free}
$\langle \rho_k\rangle<G_k$ acts freely on $\mathcal C_k$ for all sufficiently large $k$.
\end{lemma}

\begin{proof}
Suppose that $\rho^j_k[\iota]=[\iota]$ for some injection $\iota:\pi_1(\bar{S})\to\pi_1(M_k,x_k)$ and $1\le j\le k-1$. This means that there exists $g\in\pi_1(M_k,x_k)$ such that 
\[
\rho^j_k\iota(\alpha)=g\iota(\alpha)g^{-1}
\]
for every $\alpha\in\pi_1(\bar{S})$. The same remains true in $Q_k$ after composing with the embedding $\pi_1(M_k,x_k)\to Q_k=\pi_1(M-B)/\langle\langle\gamma^k\rangle\rangle$ provided by Lemma \ref{lem:van kampen}. By Lemma \ref{lem:rotation is conjugation}, inside $Q_k$ we have $\rho_k\iota(\alpha)=\gamma\iota(\alpha)\gamma^{-1}$ for every $\alpha\in\pi_1(\bar{S})$. Thus, 
\[
\gamma^j\iota(\alpha)\gamma^{-j}=g\iota(\alpha)g^{-1}
\]
or, written differently, the element $\gamma^jg^{-1}$ of $Q_k$ commutes with the subgroup $\iota(\pi_1(\bar{S}))$. 

By Lemma \ref{lem:centre}, this is possible only if $\pi_1(\bar S)$ is conjugate to a subgroup of $P_k<Q_k$, the image of $\pi_1(U-B)$ in $Q_k$. Recalling that $P_k\simeq\pi_1(B)\times\mb{Z}/k\mb{Z}$ (see Section \ref{sec:rel hyp}) and that $\pi_1(\bar S)$ is torsion-free (see Lemma \ref{lem:pi1injective}), we deduce that $\pi_1(\bar S)$ must be isomorphic to subgroup of $\pi_1(B)$. However, this cannot happen because of two facts. First, $\pi_1(\bar S)$ contains a proper subgroup (see Lemma \ref{lem:pi1injective}) isomorphic to $\pi_1(B)$. Second, it is well-known that the fundamental group of a closed hyperbolic manifold does not contain a copy of itself as a proper subgroup (by Mostow rigidity). Combined, the two facts prove that $\pi_1(\bar S)$ cannot be isomorphic to a subgroup of $\pi_1(B)$.
\end{proof} 

We now prove Theorem \ref{thmA:main}, which we restate for the convenience of the reader.

\main*

\begin{proof}
By Proposition \ref{prop:no_split}, the group $G=\pi_1(\bar{S})$ satisfies the assumptions of Corollary \ref{cor:injections}. Applying the latter, we get some $I$, independent of $k$, such that there are at most $I$ injections $\pi_1(\bar{S})\to \pi_1(M_k,x_k)$ up to conjugation and elements of $\langle\rho_k\rangle$. In other words,
\[
|\mc{C}_k/\langle\rho_k\rangle|\le I.
\]

By Lemma \ref{lem:free}, the action of $\langle\rho_k\rangle$ on $\mc{C}_k$ is free, so $\mc{C}_k$ is partitioned by such action into $\langle\rho_k\rangle$-orbits of size $k$ (the order of $\rho_k$ by Corollary \ref{cor:order_k}). In conclusion, $|\mc{C}_k|$ is divisible by $k$ and $|\mc{C}_k|/k\le I$, as required.
\end{proof}

\subsection{The action of $G_k$ on orbits and Theorem \ref{thmA:Out}}
Next, we consider injection stabilisers. We denote by 
\[
H_j:={\rm Stab}\left([\iota_j]\right)<G_k
\]
the stabiliser in $G_k$ of the conjugacy class of the inclusion $\iota_j:\pi_1(\bar{S}_j,x_k)\to\pi_1(M_k,x_k)$.

By considering $G_k$-orbits instead of the whole $\mc{C}_k$, we can refine the discussion of the previous section and prove the following.

\begin{lemma}
\label{lem:index}
Let $N>0$ be the constant provided by Corollary \ref{cor:injections} with input $G:=\pi_1(\bar{S})$. For all $j\le k$ we have 
\[
[G_k:H_j]= N_jk
\]
for some $N_j\leq N$.
\end{lemma}

\begin{proof}
We will show that the $G_k$-orbits $O$ in $\mathcal C_k$ have cardinality of the form $N_jk$ as in the statement. Once we do that, we have that the index $[G_k:H_j]$ is equal to the cardinality of the orbit $O=G_k\cdot[\iota_j]$ in $\mathcal C_k$. 

Consider an arbitrary orbit $O=G_k\cdot[\iota]$ in $\mathcal C_k$. Note that $\langle \rho_k\rangle$ acts on $O$. By Lemma \ref{lem:free}, such action is free, so that $O$ can be partitioned into $\langle \rho_k\rangle$-orbits of cardinality equal to $|\langle \rho_k\rangle|=k$. The number of such orbits is $N_O=|O|/|\langle \rho_k\rangle|\le N$. The claim follows immediately.
\end{proof}

Lemma \ref{lem:index} is the key tool to prove Theorem \ref{thmA:Out}, which we restate for convenience.

\Out*

\begin{proof}
We split the proof into two main steps, namely Claim \ref{claim:index intersection} and Claim \ref{claim:intersection trivial}.

\begin{claim}
\label{claim:index intersection}
$[G_k:\bigcap H_j]$ divides $(N!)^kk$.
\end{claim}

\begin{proof}[Proof of the claim]
We write
\[
\left[G_k:\bigcap_{j\le k}{H_j}\right]=\left[G_k:H_1\right]\cdot\left[H_1:H_1\cap H_2\right]\cdot\ldots\cdot\left[\bigcap_{j\le k-1}{H_j}:\bigcap_{j\le k}{H_j}\right].
\]
By Lemma \ref{lem:index}, the first factor is $[G_k:H_1]=N_1k$ with $N_1\le N$. The other factors have the form
\[
\left[\bigcap_{j\le r}{H_j}:\bigcap_{j\le r+1}{H_j}\right].
\]
In order to bound them, we repeatedly use the following fact: If $A$ is a finite group and $B,C<A$ are subgroups, then $[B:B\cap C]\le[A:C]$. Choosing
\[
A=\bigcap_{2\le j\le r}{H_j},\quad B=\bigcap_{j\le r}{H_j},\quad\text{\rm and}\quad C=\bigcap_{2\le j\le r+1}{H_j},
\]
we can remove $H_1$ from the intersections and get
\[
\left[\bigcap_{j\le r}{H_j}:\bigcap_{j\le r+1}{H_j}\right]\le\left[\bigcap_{2\le j\le r}{H_j}:\bigcap_{2\le j\le r+1}{H_j}\right].
\]
Applying again and again the fact in a similar fashion, we obtain  
\[
\left[\bigcap_{j\le r}{H_j}:\bigcap_{j\le r+1}{H_j}\right]\le\left[\bigcap_{2\le j\le r}{H_j}:\bigcap_{2\le j\le r+1}{H_j}\right]\le\cdots\le[H_r:H_r\cap H_{r+1}].
\]
We now claim that for every $r$ we have 
\[
[H_r:H_r\cap H_{r+1}]\leq N
\]
where $N$ is the same as in Lemma \ref{lem:index}. This is enough to conclude the proof of the claim since we showed that $[G:\bigcap{H_j}]$ is a product of $k$ integers, the first of which is $N_1k$ with $N_1\le N$ and all the others are smaller than $N$.

For simplicity of notation, we show that $[H_r:H_r\cap H_{r+1}]\leq N$ for $r=1$. 

We consider the action of $H_1$ on $G_k/H_2$. The stabiliser of $H_2\in G_k/H_2$ in $H_1$ is $H_1\cap H_2$, so that $[H_1:H_1\cap H_2]$ is the cardinality of the orbit $O:=\{h_1H_2:h_1\in H_1\}$ of $H_2\in G_k/H_2$ under $H_1$. We argue that 
\[
|O|k\le |G_k/H_2|.
\]
Since $|G_k/H_2|=N_2k$ with $N_2\le N$ (by Lemma \ref{lem:index}), this allows us to conclude $|O|\leq N_2$ and finish the proof. In order to show that $|O|k\le |G_k/H_2|$ it is enough to show that distinct points $h_1H_2,h'_1H_2\in O$ have disjoint orbits under the action of $\langle\rho_k\rangle$ on $G_k/H_2$. To this purpose, consider arbitrary $\rho,\rho'\in\langle \rho_k\rangle$. Suppose that $\rho h_1H_2=\rho'h'_1H_2$. We show that $h_1H_2=h'_1H_2$ and $\rho=\rho'$. Note that 
\[
\rho h_1H_2=\rho'h'_1H_2\Longleftrightarrow(\rho')^{-1}\rho=h'_1h_2h^{-1}_1
\]
for some $h_2\in H_2$. As outer automorphisms, the right-hand side fixes the conjugacy class of the inclusion of the intersection of $\pi_1(\bar{S}_1,x_k)$ and $\pi_1(\bar{S}_2,x_k)$, which contains $\pi_1(W_1^+,x_k)$ (by Lemma \ref{lem:pi1injective}). The exact same argument as in Lemma \ref{lem:free} gives that $\langle \rho_k\rangle$ acts freely on the conjugacy classes of said inclusion, and so we must have $\rho=\rho'$, and in turn $h_1=h'_1h_2$. Hence, $h_1H_2=h'_1H_2$ as required.
\end{proof}

\begin{claim}
\label{claim:intersection trivial}
$\bigcap H_j=\{{\rm id}\}$.
\end{claim}

\begin{proof}[Proof of the claim]
Let $[\phi]\in \bigcap H_j$. In particular, $[\phi]\in H_1$, so that, up to changing the representative in the outer class, we can assume that the restriction of $\phi$ to $\pi_1(\bar{S}_1,x_k)$ is the identity. We claim that $\phi$ is also the identity on all $\pi_1(\bar{S}_j,x_k)$ and, hence, $\phi$ is the identity because the $\pi_1(\bar{S}_j,x_k)$'s generate $\pi_1(M_k,x_k)$. 

We explain the argument for $j=2$, the other cases follow inductively along the same lines. Since $[\phi]\in H_2$ we know that the restriction of $\phi$ to $\pi_1(\bar{S}_2,x_k)$ is a conjugation by some $g\in\pi_1(M_k,x_k)$. Observe that (inside $\pi_1(M_k,x_k)$) we have
\[
\pi_1(W_2^-,x_k)=\pi_1(W_1^+,x_k)
\]
(recall that both $W_1^+=W_2^-\subset\bar{S}_1\subset M_k$ are $\pi_1$-injective by Lemma \ref{lem:pi1injective}). As the restriction of $\phi$ to $\pi_1(\bar{S}_1,x_k)$ is the identity, we have 
\[
g\alpha g^{-1}=\phi(\alpha)=\alpha
\]
for every $\alpha\in\pi_1(W_1^+,x_k)=\pi_1(W_2^-,x_k)$. So $g$ centralises $\pi_1(W_1^+,x_k)$. From Lemma \ref{lem:centre} we deduce that $g=1$ (the lemma applies similarly to the proof of Lemma \ref{lem:free}) and, hence, the restriction of $\phi$ to $\pi_1(\bar{S}_2,x_k)$ is the identity as desired.
\end{proof}

Putting together the above claims shows that $|G_k|=[G_k:\bigcap{H_j}]$ divides $(N!)^kk$. This concludes the proof of Theorem \ref{thmA:Out}. 
\end{proof}

\subsection{Fixed subgroups and Theorem \ref{thmA:Sylow}}
We now move on to the proof of Theorem \ref{thmA:Sylow}, which we restate.

\Sylow*

For the proof, we need the following definition and fact.

\begin{dfn}[Hall Subgroup]
\label{dfn:hall}
Let $G$ be a finite group and let $\Pi$ be a set of primes. A subgroup $H<G$ is a {\em Hall $\Pi$-subgroup} if the prime divisors of the order of $H$ lie in $\Pi$ and the order is coprime with the index.
\end{dfn}

By a foundational result of Wielandt \cite{Wieland}, if a finite group $G$ has a {\em nilpotent} Hall $\Pi$-subgroup $H$, then all Hall $\Pi$-subgroups of $G$ are conjugate to a subgroup of $H$. 

We can now start the proof.

\begin{proof}
Throughout this proof, assume that $k,k'$ are sufficiently large.

Suppose by contradiction that there exists an isomorphism 
\[
\psi:\pi_1(M_k)\to \pi_1(M'_{k'}).
\]
This also induces isomorphisms 
\[
\psi_{\rm aut}:{\rm Aut}(\pi_1(M_k))\to{\rm Aut}(\pi_1(M'_{k'}))\quad\text{\rm and}\quad\psi_{\rm out}:{\rm Out}(\pi_1(M_k))\to{\rm Out}(\pi_1(M'_{k'}))
\]
For simplicity we introduce the notation $G_k:={\rm Out}(\pi_1(M_k))$ and $G'_{k'}:={\rm Out}(\pi_1(M'_{k'}))$.

We now consider the rotation $\rho_k\in G_k$ and make a couple of observations. First, by Corollary \ref{cor:order_k}, if $k$ is large enough, then $\rho_k$ has order $k$ in $G_k$. In particular $k||G_k|$. By Theorem \ref{thmA:Out} applied to the GT-pair $(M,B)$, we have $|G_k||(N!)^kk$ for some fixed $N$ only depending on $(M,B)$. If $k$ has only prime factors larger than $N$ we conclude that the order and the index of $\langle\rho_k\rangle<G_k$ are coprime.

We transfer now the above considerations to the group $G'_{k'}$ (isomorphic to $G_k$). We have the following. By Theorem \ref{thmA:Out} applied to the GT-pair $(M',B')$, we have 
\[
|G'_{k'}|\left||(N'!)^{k'}k'\right.
\]
for some fixed $N'$ only depending on $(M',B')$. Since $|G_k|=|G'_{k'}|$ and $k||G_k|$ we get 
\[
k|(N'!)^{k'}k'.
\]
Hence, if $k$ has only prime factors larger than $N'$, we conclude that $k|k'$.

By Corollary \ref{cor:order_k}, if $k'$ is large enough then $\rho'_{k'}$ has order $k'$ in $G'_{k'}$. By the above considerations, as $G_k$ is isomorphic to $G'_{k'}$, it follows that the order and the index of 
\[
\langle\left(\rho'_{k'}\right)^{k'/k}\rangle<G'_{k'}
\]
are coprime. Set $\Pi$ to be the set of primes $>\max\{N,N'\}$. Observe that both
\[
\langle\psi_{\rm out}(\rho_k)\rangle\quad\text{\rm and}\quad\langle\left(\rho'_{k'}\right)^{k'/k}\rangle
\]
are (cyclic-hence-nilpotent) Hall $\Pi$-subgroups of $G'_{k'}$ (see Definition \ref{dfn:hall}). By Wielandt's Theorem \cite{Wieland}, it follows that they are conjugate to each other and, hence, that 
\[
\bar{\rho}:=\psi_{\rm out}(\rho_k)\quad\text{\rm and}\quad\bar{\rho}':=\left(\rho'_{k'}\right)^{k'/k}
\]
are conjugate, say $\bar{\rho}'=\bar{\phi}\bar{\rho}\bar{\phi}^{-1}$ for some $\bar{\phi}\in G'_{k'}$. Passing to automorphisms, we get
\[
\rho'c_g=\phi\rho\phi^{-1}
\]
where we chose
\[
\rho:=\psi_{\rm aut}(\rho_k)\quad\text{\rm and}\quad\rho':=\left(\rho'_{k'}\right)^{k'/k},
\]
where $\phi$ is an arbitrary representative of $\bar{\phi}$, and where $c_g$ denotes the conjugation by the element $g\in\pi_1(M'_{k'})$.

Note that $\phi\rho\phi^{-1}$ fixes the subgroup $\phi\psi(\pi_1(B_k))<\pi_1(M'_{k'})$.

Now recall that $\pi_1(M'_{k'})$ embeds in $Q'_{k'}=\pi_1(M'-B')/\langle\langle(\gamma')^{k'}\rangle\rangle$ (Lemma \ref{lem:van kampen}) and the composition of $\rho'_{k'}$ with this embedding coincides with the conjugation by $\gamma'$ that is $\rho'_{k'}=c_{\gamma'}$ (Lemma \ref{lem:rotation is conjugation}).

Thus, setting $\delta:=(\gamma')^{k'/k}$, we have $\rho'c_g=c_\delta c_g=c_{\delta g}$ on the image of $\pi_1(M'_{k'})$. In particular, the fixed subgroup of $c_{\delta g}$ is the centraliser in $\pi_1(M'_{k'})$ of $\delta g$. By the above discussion, $c_{\delta g}=\phi\rho\phi^{-1}$ on $\pi_1(M'_{k'})$ so the fixed subgroup in $\pi_1(M'_{k'})$ of $c_{\delta g}$ is also equal to the fixed subgroup of $\phi\rho\phi^{-1}$ which is $\phi\psi(\pi_1(B_k))$.

As $c_{\delta g}=\phi\rho\phi^{-1}$ on $\pi_1(M'_{k'})$ and $\phi\rho\phi^{-1}$ has order $k$ in $G'_{k'}$, we deduce that $(\delta g)^k\in Q'_{k'}$ centralises $\pi_1(M'_{k'})$ and, hence, $(\delta g)^k=1$ by Lemma \ref{lem:centre}. By Lemma \ref{lem:filling_finite_order}, we must have that $\delta g$ is conjugate to a power of $\gamma'$ in $Q'_{k'}$, say 
\[
\delta g=h(\gamma')^rh^{-1}
\]
for some $h\in Q'_{k'}$. We now recall that $\delta=(\gamma')^{k'/k}$. We also recall that $\pi_1(M'_{k'})$ is normal in $Q'_{k'}$ and that the quotient $Q'_{k'}/\pi_1(M'_{k'})$ is an order $k'$ cyclic subgroup generated by $\gamma'$ (see Lemmas \ref{lem:van kampen} and \ref{lem:rotation is conjugation}). Projecting to the quotient the above relation, we get that $r=k'/k$ and, hence, we have $\delta g=h\delta h^{-1}$. 

Thus, the centraliser $C(\delta g)$ of $\delta g$ in $Q'_{k'}$ is the centraliser of $h\delta h^{-1}$, which, in turn, is the conjugate by $h$ of the centraliser $C(\delta)$ of $\delta$ in $Q'_{k'}$, that is, $C(\delta g)=hC(\delta)h^{-1}$. The centraliser of $\delta g$ in $\pi_1(M'_{k'})$ is the intersection
\[
C(\delta g)\cap\pi_1(M'_{k'})=hC(\delta)h^{-1}\cap\pi_1(M'_{k'})=h\left(C(\delta)\cap\pi_1(M'_{k'})\right)h^{-1}
\]
where the latter equality follows from the fact that $\pi_1(M'_{k'})$ is normal in $Q'_{k'}$. In particular, as the conjugation by $h$ induces an automorphism of $\pi_1(M'_{k'})$, the subgroup $C(\delta g)\cap\pi_1(M'_{k'})$ is isomorphic to $C(\delta)\cap\pi_1(M'_{k'})$. The latter is equal to $\pi_1(B'_{k'})$.

We are now able to conclude. On the one hand, we just proved that the centraliser of $\delta g$ in $\pi_1(M'_{k'})$ is isomorphic to $\pi_1(B'_{k'})$. On the other hand, we know that the centraliser of $\delta g$ in $\pi_1(M'_{k'})$ is equal to the fixed subgroup of $\phi\rho\phi^{-1}$, which is $\phi\psi(\pi_1(B_k))$. Thus $\pi_1(B_k)=\pi_1(B)$ is isomorphic to $\pi_1(B'_{k'})=\pi_1(B')$. By Mostow rigidity, recalling that $d\ge 5$, it follows that $B'$ is isometric to $B$. This contradicts our initial assumptions and finishes the proof.
\end{proof}

\bibliographystyle{alpha}
\bibliography{biblio}

\end{document}